\newif\iffinal
\finaltrue	

\documentclass[11pt,reqno]{amsart}

\iffinal\else\usepackage[notref,notcite]{showkeys}\fi

\usepackage{xifthen}

\usepackage{subfig}
\usepackage{amsmath}
\usepackage{amssymb, latexsym, amsfonts, amscd, amsthm, mathrsfs}
\usepackage[usenames,dvipsnames]{color}
\usepackage[all]{xy}
\usepackage{graphicx}
\usepackage{epsfig}
\usepackage[colorlinks=true,linkcolor=blue,citecolor=red]{hyperref}
\usepackage{amssymb}
\usepackage{enumerate}
\usepackage{bm}
\usepackage{enumerate}

\newenvironment{enumeratea}{\begin{enumerate}[\upshape (a)]}{\end{enumerate}}

\usepackage{todonotes}
\def\sss{\scriptscriptstyle}

\numberwithin{equation}{section}

\definecolor{purple}{rgb}{0.9,0,0.8}

\definecolor{gray}{rgb}{0.7,0.7,0.7}

\newtheorem{thm}{Theorem}[section]
\newtheorem{theorem}{Theorem}[section]

\newtheorem{lem}[thm]{Lemma}

\newtheorem{open}[thm]{Open Problem}
\newtheorem{dfn}[thm]{Definition}

\theoremstyle{definition}
\newtheorem{ex}[thm]{Example}
\newtheorem{rmk}[thm]{Remark}
\newtheorem{rem}[thm]{Remark}
\newtheorem*{rmk*}{Remark}

\newcommand{\beq}{\begin{equation}}
\newcommand{\eeq}{\end{equation}}

\newcommand{\eps}{\varepsilon}

\newcommand{\E}{\mathbb{E}}

\newcommand{\cB}{\mathcal{B}}

\newcommand{\cE}{\mathcal{E}}

\newcommand{\cK}{\mathcal{K}}
\newcommand{\cL}{\mathcal{L}}

\newcommand{\cR}{\mathcal{R}}

\newcommand{\cX}{\mathcal{X}}

    \newcommand{\gX}{\mathfrak{X}}
    \newcommand{\gY}{\mathfrak{Y}}

\newcommand{\f}{\frac}

\newcommand{\set}[1]{\{#1\}}

\newcommand{\vX}{\mathbf{X}}\newcommand{\vY}{\mathbf{Y}}

\newcommand{\vu}{\mathbf{u}}
\newcommand{\vx}{\mathbf{x}}
\newcommand{\vy}{\mathbf{y}} 

\newcommand{\mvu}{\boldsymbol{u}}
\newcommand{\mvx}{\boldsymbol{x}}\newcommand{\mvy}{\boldsymbol{y}}
\newcommand{\mvz}{\boldsymbol{z}}

\newcommand{\mvalpha}{\boldsymbol{\alpha}}\newcommand{\mvbeta}{\boldsymbol{\beta}}

\newcommand{\mvtheta}{\boldsymbol{\theta}}

\newcommand{\bR}{\mathbb{R}}

\newcommand{\prob}{\mathbb{P}}

\begin{document}

\title[Weighted exponential random graph models]{Weighted Exponential Random graph models: Scope and large network limits}

\date{}

\author[Bhamidi]{Shankar Bhamidi$^1$}
\author[Chakraborty]{Suman Chakraborty$^1$}
\address{$^1$Department of Statistics and Operations Research, 304 Hanes Hall, University of North Carolina, Chapel Hill, NC 27599}
\author[Cranmer]{Skyler Cranmer$^2$}
\address{$^2$Department of Political Science, The Ohio State University, 2032 Derby Hall, 154 North Oval Mall, Columbus, OH 43210}
\author[Desmarais]{Bruce Desmarais$^3$}
\address{$^3$Department of Political Science, Pennsylvania State University, 321 Pond Lab,
University Park, PA 16802}
\email{bhamidi@email.unc.edu, sumanc@live.unc.edu, cranmer.12@osu.edu, bdesmarais@psu.edu}

\maketitle

\begin{abstract}
We study models of weighted exponential random graphs in the large network limit. These models have recently been proposed to model weighted network data arising from a host of applications including socio-econometric data such as migration flows and neuroscience.  Analogous to fundamental results derived for standard (unweighted) exponential random graph models in the work of Chatterjee and Diaconis, we derive limiting results for the structure of these models as the number of nodes goes to infinity. Our results are applicable for a wide variety of base measures including measures with unbounded support. We also derive sufficient conditions for continuity of functionals in the specification of the model including conditions on nodal covariates.  
Finally we include a number of open problems to spur further understanding of this model especially in the context of applications. 
\end{abstract}

\section{Introduction}

Exponential random graph models (ERGMs) constitute one of the fundamental tools in the statistical analysis of networks \cite{holland1981exponential,wasserman1996logit,robins2007introduction,snijders2006new}. We start by describing ERGMs as these serve as the stepping stone for the model studied in this paper. Write $[n]:=\set{1,2,\ldots, n}$ and let

 \[\gX_n:=\set{(y_{ij})_{i\neq j\in [n]}: y_{ij} = y_{ji}, \;\; y_{ij}\in \set{0,1}},\] 
 denote the set of all \emph{unweighted} and undirected simple networks on $n$ vertices; here the presence of an edge between vertices $i,j$ is represented by $y_{ij} =1$ and $y_{ij} =0$ otherwise. Given an observed network $\vY:= (Y_{ij})_{i\neq j\in [n]} \in \cX_n$, one can model this observed network as a sample from a probability distribution of the form,
 \begin{equation}\label{thu:102}
 \prob_n(\vy, \mvtheta)= \frac{\exp(\mvtheta^\prime T(\vy))}{\sum_{\vx \in \gX_n} \exp(\mvtheta^\prime T(\vx))} , \qquad \vy \in \gX_n.\end{equation}
 Here for $m\geq 1$, $\mvtheta\in \bR^m$ is a parameter vector whilst $T(\cdot):\gX_n \to \bR^m$ is a vector of statistics often constructed with domain-specific motivations and knowledge and might include terms that measure clustering and reciprocity in the network, other notions of connectivity and might include node-level covariates. The aim or object of inference then from the observed network $\vY$ is this parameter vector $\mvtheta$.

  There is an enormous amount of methodology using ERGMs built for unweighted network models.  In the last few years, predominantly motivated by applications and real-world data, a host of weighted network data have arisen ranging from financial applications \cite{iori2008network}, modeling migration flows between different regions \cite{chun2008modeling,desmarais2012statistical} and brain networks \cite{simpson2011exponential}. While the development of generative models analogous to ERGMS in this context is much less developed, there has been recent progress both in methodological developments \cite{robins1999logit,krivitsky2012exponential,desmarais2012statistical} as well as computational tools to generate and derive statistical inference for these models \cite{conjec2str}. Following \cite{desmarais2012statistical}, we will refer to this general family of models as Generalized Exponential Random graph models (GERGM); a precise definition of the model is given in Section \ref{sec:results}; for closely associated settings, see \cite{krivitsky2012exponential} for model formulations and terminology that further inspired this work; for the sake of concreteness we stick with the term GERGM. Recently \cite{yin2016phase,demuse2017phase} have investigated these models when the edge weights are bounded. The aim of this work is to study this model in full generality. To fix ideas, next we informally describe the GERGM model. Let 
  \[\gY_n:=\set{(y_{ij})_{i\neq j\in [n]}: y_{ij} = y_{ji}, \;\; y_{ij}\in \mathbb{R}},\] 
  denote the set of weighted and undirected simple networks on $n$ vertices; here $y_{ij}$ represents the weight of the edge between vertices $i, j$. We do not distinguish between an edge with zero weight and an missing edge, we write $y_{ij} =0$ in either cases. Suppose $Q$ is a real valued random variable with density $q$ (with respect to the Lebesgue measure). Consider the following probability distribution on $\gY_n$
  
   \begin{equation}\label{thu216} \prob_n(\vy, \mvtheta) \propto {\exp(\mvtheta^\prime T(\vy))} \prod_{1\leq i<j \leq n} q(y_{ij}), \qquad \vy \in \gY_n.\end{equation}
Similar to ERGM, here also for $m\geq 1$, $\mvtheta\in \bR^m$ is a parameter vector and $T(\cdot):\gX_n \to \bR^m$ is a vector of statistics that captures the dependence between edges. Note that $\prob_n$ in \eqref{thu216} is a probability distribution only when $\mathbb{E}\left(\exp(\mvtheta^\prime T(\vY))\right)$ is finite, where $\set{Y_{ij}:1\leq i<j\leq n} $ is an array of independent  and identically distributed random variables with common distribution $Q$. Here $Q$ captures the inherent nature of edge-weights if there was no dependence. For example, if $Q \sim$ Bernoulli$(1/2)$, then this gives usual ERGM. $Q$ can be discrete or continuous depending on the application domain. For example if the edge-weights are non-negative integers (possibly unbounded), then a possible choice of $Q$ would be Poisson distribution with appropriate parameter, one of the examples considered in \cite{krivitsky2012exponential}. 

A common problem encountered by practitioners regarding the ERGMs is the problem of degeneracy. Informally it says that for some choices of the sufficient statistic $T$, for most of the parameter values, the probability distribution given by \eqref{thu:102} place all of its mass on either empty graphs or complete graphs (see \cite{snijders2006new}). Motivated by the empirical study in \cite{conjec2str} we investigate degeneracy phenomenon for GERGM. In particular we show that when the edge weights are standard normal distribution the GERGM edge-two-star model does not suffer from degeneracy.

In Section \ref{sec:results} we will explicitly state our assumptions and findings. Before diving into the precise statement of the model and results, let us informally outline the aims of this paper.
  
  \noindent {\bf Aims:} We had five major goals in writing this paper: 


\begin{enumeratea}
	\item Develop general theory in order to derive a formula for the limiting normalizing constant of GERGM and understand the behavior of GERGM when $n$ (number of nodes) goes to infinity with minimal continuity assumptions on the specifications (the functional $T$) based on large deviation results for symmetric random matrices \cite{ldprm} that deals with general (possibly unbounded) GERGM specifications. 
	\item Via both direct calculations as well as application of the general theory developed above, show that various base measures including the normal distribution as well as distributions with density proportional to $\exp{(-x^4)}$ satisfy the conditions required for the main results. These calculations are driven by ``proof of concept'' motivations and can serve as the starting point to illustrate the kinds of calculations required for other base measures. 
	\item Derive concrete expressions of the limits for various specifications involving ``homomorphism'' counts of motifs. 
	\item Investigate ``degeneracy" phenomenon for GERGM models. In particular we show that if the edge weights are standard normal distribution then the GERGM edge-two-star model does not suffer from degeneracy.
	\item Understand issues regarding continuity (or lack thereof) of standard functionals in the general weighted case and in particular make a start in understanding the scope of GERGM specifications.  
\end{enumeratea}

We have also included a number of open problems that we hope will motivate further work on this model. 

\subsection{Organization of the paper}
We start with some preliminary notation required to setup the formulation of the model in Section \ref{sec:graph-lim-ldp}. In Section \ref{sec:results} we describe the model and the main results. Section \ref{sec:grr-hom-res} discusses the scope of this model with reference to continuity considerations for the specification.  In Section \ref{sec:disc} we describe the relationship of this paper to existing results and literature as well as describe a number of open directions.  Finally Section \ref{sec:proof} contains the proofs of all our main results. 
\section{Graph limits and Large deviation preliminaries}
\label{sec:graph-lim-ldp}
In this section we start with some notation required to define our model and results. Much of the exposition below follows \cite{lngl}. 
A symmetric $n \times n$ random matrix $X(\omega) = \{x_{ij}\}_{1\leq i,j\leq n}$ with $x_{ij} = x_{ji}$ for all $i,j\in [n]$  can be mapped to a symmetric kernel in the following obvious manner:
\begin{equation}\label{mapintokernel}
k(x,y, \omega) = \sum_{i,j =1}^n{x_{ij}(\omega)\bold{1}_{J_i^n}(x) \bold{1}_{J_j^n}(y)},
\end{equation}
where $J_1^n = [0, \f{1}{n}]$ and for $i=2,\ldots,n$, $J_i^n$ is the interval $(\frac{i-1}{n}, \frac{i}{n}]$. Thus a probability measure on the space of symmetric random matrices results in a family of induced probability measures $Q_n$ on $D= [0,1] \times [0,1]$. $\mathcal{K}$ is defined as the space of symmetric measurable functions from $D\rightarrow \bR$. For each fixed $n$, the range of the map $X \rightarrow k$ in \eqref{mapintokernel} is a finite-dimensional subspace of step functions $\mathcal{K}_n \subset \mathcal{K}$. The cut distance \cite{lngl,gl1,gl2,gl3} in the space $\mathcal{K}$ is defined as follows,
\begin{equation}
\label{eqn:dist-def}
	d(k_1, k_2) = \sup_{|\phi|\leq 1, |\psi| \leq 1}\left|{\int(k_1(x,y) - k_2(x,y))\phi(x) \psi(y)\,dx\,dy}\right|,
\end{equation}
where $\phi$ and $\psi$ are Borel measurable functions on $[0,1]$. One can equivalently write
$$
d(k_1, k_2) = \sup_{A, B \subset [0,1]}{\left|\int_{A \times B}(k_1(x,y) - k_2(x,y)) \,dx\,dy\right|},
$$
where $A$ and $B$ are Borel subsets of $[0,1]$. We will quotient the space via the following equivalence  relation: Write $\Sigma$ as the space of all measure preserving bijections (with respect to the Lebesgue measure) $\sigma: [0,1] \rightarrow [0,1]$. For $k_1, k_2\in \cK$, say that $k_1 \sim k_2$ if, 
\[k_1(x,y) = \sigma k_2(x,y) : = k_2(\sigma x,\sigma y), \qquad \mbox{a.e. } x,y, \qquad \mbox{for some } \sigma \in \Sigma. \]  
We denote the orbit $\{\sigma k:\sigma \in \Sigma \}$ by $\tilde{k}$. Write $\tilde{\cK}:=\cK / \sim$ for the quotient space under the relation $\sim$ on $\mathcal{K}$ and $\tau$ for the natural map from $k \rightarrow \tilde{k}$. Since the distance $d$ in \eqref{eqn:dist-def} is invariant under $\sigma$, one can define a natural distance $\delta$ on $\tilde{\mathcal{K}}$ via
\begin{equation}
\label{eqn:delta-def}
	\delta(\tilde{k}_1, \tilde{k}_2) = \inf_{\sigma}{d(\sigma k_1, k_2)} = \inf_{\sigma}{d( k_1, \sigma k_2)} = \inf_{\sigma_1, \sigma_2}{d(\sigma_1 k_1, \sigma_2 k_2)},
\end{equation}
making $(\mathcal{\tilde{K}}, \delta)$ into a metric space.   \par
 Write $\tilde{Q}_n$ for the measure on $\tilde{\cK}$ obtained as the push-forward of the measure $Q_n$ as above on $\cK$ i.e. $\tilde{Q}_n (\tilde{S}) = Q_n(\tau^{-1}(\tilde{S}))$ for all measurable $\tilde{S} \subset \tilde{\mathcal{K}}$. We now introduce some notation to state a large deviation result for $\tilde{Q}_n$ proved in \cite{ldprm}. We assume that $x_{ij}$'s are independent and identically distributed with measure $\mu$ and further  
\begin{equation}\label{A1}
\int{e^{\theta x^2}}\,\mu(dx) < \infty,
\end{equation}
for all $\theta \in \bR$.
This implies in particular that the moment generating function $M(\theta) := \int_{\bR}{e^{\theta x}} \,\mu(dx)$ is finite and satisfies
\begin{equation}\label{con1A1}
\limsup_{|\theta|\rightarrow \infty}{\frac{1}{\theta^2} {\ln M(\theta)}} = 0.
\end{equation}
Thus the conjugate rate function of Cramer defined via    
\begin{equation}\label{con2A22}
h(x) := \sup_{\theta}{[\theta x - \ln M(\theta)]},
\end{equation}
satisfies
$$
\liminf_{|x| \rightarrow \infty}{\frac{h(x)}{x^2}} = +\infty.
$$
Now we define the rate function $I(\cdot)$ on $\mathcal{K}$ as follows
\begin{equation}\label{ratefunction}
I(k) = \frac{1}{2} \iint_{D}{h(k(x,y))} \, dx \,dy, \qquad k\in \cK.
\end{equation}

Note that the rate function is invariant under measure preserving bijection and thus it extends naturally to a rate function on $\mathcal{\tilde{K}}$ naturally. The following result was proven in \cite{ldprm}

\begin{theorem}[{\cite{ldprm}}]\label{thm:cha-varadhan}
	\label{thm:cha-vara}
	Under assumption (\ref{A1}), the sequence of measures $\set{\tilde{Q}_n}_{n\geq 1}$ satisfies a large deviation property with rate function $I(\cdot)$, that is, for every closed $C \subset \tilde{\mathcal{K}}$,
\begin{equation}\label{ldpclosed}
	\limsup_{n \rightarrow \infty}{\frac{1}{n^2} {\ln \tilde{Q}_n(C)}} \leq - \inf_{\tilde{k} \in C}{I(\tilde{k})},
	\end{equation}
	and for open $U \subset \tilde{\mathcal{K}}$,
	\begin{equation}\label{ldpopen}
	\liminf_{n \rightarrow \infty}{\frac{1}{n^2} {\ln \tilde{Q}_n(U)}} \geq - \inf_{\tilde{k} \in U}{I(\tilde{k})}. 
	\end{equation}
	
\end{theorem}

\section{Generalized exponential random graph Model}
\label{sec:results}

\subsection{Model Formulation and Main Theorem}

We now formally describe the main model of interest for this paper. In words, this model is obtained by naturally tilting a base measure made up of independent entries via an appropriate specification. Let us describe each of these ingredients. 
 
\begin{dfn}[Base measure]
Let $n\geq 1$ be a fixed positive integer.  Let $P_n$ is the measure of the i.i.d. random variables $(q_{ij})_{1\leq i<j \leq n}$, where $q_{ij}$ is the edge weight of the edge $\{i,j\}$, for ${1\leq i<j \leq n}$. We assume $q_{ij} = q_{ji}$ for ${1\leq i<j \leq n}$ and, for the diagonal, let $q_{ii} =\delta_0$ be the unit mass at zero, for $i = 1,\ldots,n$ independent of the remaining edges. We map the matrix $(q_{ij})_{1\leq i,j \leq n}$ into an element of $\cK$ using \eqref{mapintokernel}. Let $Q_n$ be the measure obtained by the induced measure on $\cK$ and, $\tilde{Q}_n$ be the corresponding push-forward measure on $\tilde{\cK}$. 
Call $\tilde{Q}_n$ the $\bold{base}$ $\bold{measure}$.
\end{dfn}
Next we formally define the generalized exponential random graph model.

\begin{dfn}[GERGM]
Fix a function $T:\mathcal{\tilde{K}} \to \bR$. Then the generalized exponential random graph model \cite{desmarais2012statistical,krivitsky2012exponential} is a probability measure $\tilde{R}_n$ on $\tilde{\cK}$ defined via tilting the base measure $\tilde{Q}_n$ using $T$. Formally,

\begin{equation}\label{ermmdfn}
\, d \tilde{R}_n(\tilde{k}) = \exp\{n^2( T(\tilde{k}) - \psi_n)\} \, d\tilde{Q}_n(\tilde {k}), \qquad \tilde{k}\in \tilde{\cK}.
\end{equation}
\end{dfn}
\begin{rmk}
The base measure models the weights of the edges. For example, if the edge weights in a network are non-negative then the $q_{ij}$'s should be taken as non-negative random variables. The function $T$ controls the dependence between the edges. $T$ is usually a function of the graph under consideration. For example, in ERGM, choices of $T$ include (appropriately scaled) number of edges, triangles, two-stars etc. Although the definition of GERGM is valid for any function $T$, in this paper, we will focus on some popular choices of $T$ that includes weighted homomorphism densities of triangles, two-stars etc. These are discussed in detail in section \ref{sec:grr-hom-res}.
\end{rmk}

  The initial goal of this Section is to understand the asymptotics for the partition function or normalizing constant $\psi_n$. Later sections develop the ramifications of this asymptotics. First we need some further notation.  
  From \eqref{ermmdfn} it is easy to see that,
\begin{equation} \label{ermm}
\psi_n  = \frac{1}{n^2} \ln{\int_{ \mathcal{\tilde{K}}}{e^{n^2T(\tilde{k})}} \, d \tilde{Q}_n(\tilde{k})}.
\end{equation} 
We need the following truncation operator on $\bR$. Note that for any $q\in \bR$ and fixed $l>0$, one may decompose $q = f_l(q)+ g_l(q)$ where,
\begin{align}
f_l(q) & = q \text{         if   } |q|\leq l, \notag \\
&= l \text{          if   } q\geq l, \notag\\
&= -l \text{   if   } q\leq -l. \label{eqn:trunc}
\end{align}
Obviously  $g_l(q) = q - f_l(q)$. This decomposition extends naturally (when applied entry-wise) to $\cK$ and thus to $\tilde{\cK}$. We write the corresponding decomposition of $k = f_l(k) + g_l(k)$ for $k \in \mathcal{{K}}$ and analogously for $\tilde{k}\in \tilde{\cK}$. Write $\mathcal{K}^l = \{k \in \mathcal{K}: |k(x,y)|\le l\}$. We will need some smoothness assumptions on the function $T$: 

{\bf (C1):} Suppose that for each fixed $l> 0$,  $T$ is a bounded continuous function in cut-metric when restricted to $\mathcal{K}^l $ and further satisfies 
\begin{equation}\label{assT}
\int_{ \tilde{\mathcal{K}}}{\exp(n^2 T(\tilde{k}))} \, d \tilde{Q}_n(\tilde{k}) < \infty,
\end{equation}

{\bf (C2):} $T(\tilde{k}_0) - I(\tilde{k}_0) \geq 0$, for some  fixed $\tilde{k}_0 \in \mathcal{\tilde{K}}$ and further, given any $\epsilon> 0$, there exists fixed $\epsilon'=\epsilon^\prime(\epsilon) >0$ such that, 
\begin{equation}\label{assneglect}
\limsup_{l \rightarrow \infty}\limsup_{n \rightarrow \infty}\f{1}{n^2} \ln \int_{ \{T(\tilde{k}) - T(f_l(\tilde{k})) \geq \epsilon\}} e^{n^2T(\tilde{k})} \, d\tilde{Q}_n(\tilde{k}) \leq -\epsilon',
\end{equation}

{\bf (C2)$^\prime$:} 
Given any $\epsilon>0$,
\begin{equation}\label{assneglect1}
\limsup_{l \rightarrow \infty}\limsup_{n \rightarrow \infty}\f{1}{n^2} \ln \int_{ \{T(\tilde{k}) - T(f_l(\tilde{k})) \geq \epsilon\}} e^{n^2T(\tilde{k})} \, d\tilde{Q}_n(\tilde{k}) = -\infty.
\end{equation}

\begin{thm}\label{normconst}
  Assume condition {\bf (C1)} and either {\bf (C2)} or {\bf (C2)$^\prime$}. Then we have the following evaluation of the normalizing constant: 
\begin{equation}\label{variaeqn}
\psi = \lim_{n \rightarrow \infty}{\psi_n} = \lim_{l \rightarrow \infty}\sup_{\tilde{k} \in \mathcal{\tilde{K}}^l}{(T(\tilde{k}) - I(\tilde{k}))}
\end{equation}
Here $I$ is the rate function defined in (\ref{ratefunction}).
\end{thm}

\begin{rem}
	Note that \eqref{assneglect1} implies \eqref{assneglect}. It is sufficient to assume \eqref{assneglect1} to prove the Theorem \ref{normconst} but we have used the fact $T(\tilde{k}_0) - I(\tilde{k}_0) \geq 0$, for some  fixed $\tilde{k}_0 \in \mathcal{\tilde{K}}$ together with \eqref{assneglect} to prove Theorem \ref{normconst}.
\end{rem}

Our next goal is to establish asymptotic behavior of a typical realization from the GERGM model as $n\to\infty$. 
 We will prove that as long as the maximizers of $T(\tilde{k})-I(\tilde{k})$ are finite we will be able to say that a typical graph will concentrate around the maximizers, where the distance is measured in cut-metric as defined in \eqref{eqn:delta-def}. This allows one to extract important information about  graph properties that are continuous in cut-metric, including homomorphism densities of small subgraphs. Let us denote $\tilde{F}^*$  as the set of maximizers of $T(\tilde{k})-I(\tilde{k})$.  By  \cite{gl9}, $\mathcal{\tilde{K}}^l$ is a compact in $\mathcal{\tilde{K}}$ for each $l>0$. Further $I$ is lower semi-continuous by \cite{ldprm}. Also we assume that the maximizers of $T(\tilde{k}) - I(\tilde{k})$ lie in the set $\tilde{\mathcal{K}}^{l_0}$ for some $l_0>0$.

\begin{thm}\label{ermmconcentration}
Consider the distribution $R_n$ as in \eqref{ermmdfn} and assume that the functional $T$ satisfies condition {\bf (C1)} and either {\bf (C2)} or {\bf (C2)}$^{\prime}$ so that the assertion of Theorem \ref{normconst} holds. Further assume that there exists $l_0 >  0$ such that, the maximizers of $T(\tilde{k}) - I(\tilde{k})$ are in the set $\tilde{\mathcal{K}}^{l_0}$.
 Then for any $\eta>0$ there exist constants $C,\gamma >0$ such that, 
\begin{equation}\label{eqn:ermmconcentration}
R_n(\delta(\tilde{k}, \tilde{F}^*) \ge \eta) \le C e^{-n^2 \gamma}
\end{equation}
for all $n \geq 1$.
\end{thm}

\begin{rmk}
Theorem \ref{ermmconcentration} suggests that the GERGM model defined in \eqref{ermm} concentrates around the set of maximizers of $T(\tilde{k})-I(\tilde{k})$ when the number of nodes are growing large, we will use this fact to establish the properties of some popular GERGM models. 
\end{rmk}
Suppose that $H$ is an unweighted  graph, let $V(H)$ be the set of vertices and $E(H)$ be the set of edges of $H$. Also let $|E(H)| =e(H)$. The homomorphism density of a graph $H$ into a kernel $k$ is denoted by $t(H, k)$ and is given as 
\begin{equation}\label{homdensimple}
t(H, k) = \int_{[0, 1]^{|V(F)|}}   {\prod_{(k,l) \in E(H)}k(x_k, x_l)} \, \prod_{k \in V(F)}{dx_k}.
\end{equation}
\begin{rmk}
A typical homomorphism density of a graph $H$ into $G$ can be written in the same way as \eqref{homdensimple} where the kernel $k$ in \eqref{homdensimple} is formed by mapping $G$ into a kernel using \eqref{mapintokernel}. Also note that $t(H, k) = t(H, \tilde{k})$. We discuss this in more detail in Section \ref{sec:grr-hom-res}. 
\end{rmk}

Suppose that $H_1$ is a graph with two vertices and a single edge joining these vertices and $H_i$'s are graphs with at least two edges for $2\leq i\leq s$. The following theorem gives us the concentration of a typical GERGM model where 
\begin{equation}\label{gergm}
T(k) = \sum_{i=1}^s{\beta_i t(H_i, k)},
\end{equation}
for $\beta_i \geq0$ for all $i\geq2$. It is known that $t(H_i, k)$'s are continuous in cut-metric \cite{gl2} when restricted to $\cK^l$ for some $l>0$, thus linear combination of them is also continuous.We say that a parameter set $\mvbeta:= (\beta_1, \ldots, \beta_s)$ for a collection of statistics $(H_1, H_2, \ldots, H_k)$ is admissible if $T(\cdot)$ satisfies \eqref{assT} and either {\bf (C2)} or {\bf (C2)}$^{\prime}$. Finally, for $u\in \bR$ define $k(x,y) = u$ for all $x,y\in [0,1]$ and let $\tilde{k}^u$ denote the image of this constant function in $\tilde{\cK}$.

\begin{thm}
	\label{thm:spec-gergm}
Consider the GERGM model with  $T(k) = \sum_{i=1}^s{\beta_i t(H_i, k)}$ where $\mvbeta:=(\beta_1,\ldots, \beta_s)$ are admissible and $\beta_2,\ldots, \beta_s$ are non negative real numbers. Also suppose either the kernel $k$ is non-negative or $e(H_i)$'s are even positive integers for all $2\leq i \leq s$. Then the value of the normalizing constant is given by
\begin{equation}\label{stdnormconst}
\lim_{n \rightarrow \infty}{\psi_n} = \sup_{u }{\left(\sum_{i=1}^s \beta_i u^{e(H_i)} - I(u)\right)}
\end{equation}
Let $K$ be the set of maximizers of the function $g(\cdot)$ defined via $g(u):= {\sum_{i=1}^s \beta_i u^{e(H_i)} - I(u)} $. Assume that,
\begin{equation}\label{finitemaxermm}
\lim_{|u| \rightarrow \infty}{\sum_{i=1}^s \beta_i u^{e(H_i)} - I(u)} = -\infty.
\end{equation}
 Then $K$ has finitely many elements and 
\begin{equation}\label{hightemp}
\min_{u \in K}\delta(\tilde{k}_n, \tilde{k}^u) \rightarrow 0.
\end{equation}
almost surely.
\end{thm}

\begin{rmk}
Note that there might be one or more solution to the variational problem \eqref{stdnormconst}. In \cite{shankar} the authors described the subset of parameter regime $\beta_i\geq 0$ for $i\geq 2$ as ``high-temperature regime" when $\sum_{i=1}^s \beta_i u^{e(H_i)} - I(u)$ has an unique maximizer. Also in the context of unweighted graph \cite{cha} defined the parameter regime to be replica symmetric phase when all the maximizers of $T(k) - I(k)$ are constant functions.  
\end{rmk}

We have already seen the limiting normalizing constant captures important information about the model. Our next result states it is in fact a continuous function of $\mvbeta$. We will use the following Theorem later to obtain result on the ``degeneracy phenomenon". 
\begin{thm}\label{thm:partfuncont}
Consider the GERGM model with  $T(k) = \sum_{i=1}^s{\beta_i t(H_i, k)}$, where $e(H_i)$'s are positive even integers for all $2\leq i \leq s$ and the base measure is supported on the whole real line,  satisfying \eqref{finitemaxermm} for all $\mvbeta \in B$ for some $B$, where $B$ is an open subset of $\bR \times \bR_+ \times \ldots \times \bR_+$. 
Further assume that for all $\mvbeta \in B$, the GERGM model as defined above is admissible and $\beta_2,\ldots, \beta_s$ are non negative real numbers. Then the limiting normalizing constant is continuous in $\mvbeta$.  \end{thm}

\begin{rmk}\label{rmk:thm:partfuncont}
If all the assumptions of theorem \ref{thm:partfuncont} holds for $\beta \in D$, where $D \subset \bR \times \bR_+ \times\ldots \bR_+$ is compact, then the limiting normalizing constant is uniformly continuous in $D$.
\end{rmk}
One of the problems commonly encountered in the context of exponential random graph models is ``degeneracy". We briefly explain this phenomenon in the context of a specific model, the so called edge-triangle model and refer the interested reader to \cite[Section 5]{cha} for an extended discussion and references. Consider the edge-triangle-model given by, 
\[T(k) = \beta_1t(H_1, k) +\beta_2 t(H_3, k)\]
where $H_1$ is a single edge and $H_3$ is a triangle i.e. a complete graph on three vertices. Now under the assumptions of Theorem \ref{thm:spec-gergm} one can compute the limiting normalizing constant when $\beta_2>0$. It is given by $\sup_{u}(\beta_1u + \beta_2 u^2 -I(u))$. Denote the maximizer(s) by $u^*(\beta_1, \beta_2)$. For base measure Bernouli(1/2), it was shown in \cite{cha} that when $\beta_1$ is negative and below a certain threshold $u^*(\beta_1, \beta_2)$ becomes discontinuous in $\beta_2$. In particular, the variational problem had two maximizers at the point of discontinuity. This phase transition phenomenon was investigated in \cite{snijders2006new} and first rigorously proved for the edge-triangle model in \cite{cha}. In particular they \cite{cha} showed some parameter values the typical edge-density of this model is very small and varying the parameter slightly makes the edge-density much higher (see \cite[Theorem 5.1]{cha} for the rigorous statement). This phase transition phenomenon is called degeneracy in \cite{snijders2006new}.

What do our results say about this notion of degeneracy in the general context of GERGM where one can have general weighted base measures? Switching gears, we now study a particular model of relevance in applications in \cite{conjec2str}, the so-called edge-two-star model given by,
\[T(k) = \beta_1t(H_1, k) +\beta_2 t(H_2, k),\]
where $H_2$ is a two-star or a triangle with one deleted edge. This model was analyzed \cite{conjec2str} with truncated normal distribution as the base measure. It was suggested ( \cite[Figure 5]{conjec2str} and related discussion) that the edge-two-star model does not suffer from degeneracy. More precisely the simulation results in \cite{conjec2str} suggested that the edge density is continuous function of $\beta_2$ when the $\beta_1$ is set equal to $-2$. Later we will prove that when the base measure is standard normal distribution this claim is indeed true. For a general class of models that includes edge-two-star model, our next theorem provides an avenue to detect regions of the parameter space that do not suffer from degeneracy. Informally for the edge-two-star (for suitable base measure) model the theorem states that in a parameter region where the variational problem $\sup_{u}(\beta_1u + \beta_2 u^2 -I(u))$ is uniquely maximized the model does not suffer degeneracy. 
\begin{thm}\label{continhightemp}
We work under the same model assumptions as Theorem \ref{thm:partfuncont}. Denote $l(\mvbeta,u) := \sum_{i=1}^s \beta_i u^{e(H_i)} - I(u)$ and $l_{max}(\mvbeta)=\sup_{u}l(\mvbeta,u)$. If for a $\mvbeta^0 \in B$ there is a finite open set $O_ {\beta^0}$ such that $\bar{O}_ {\mvbeta^0} \subset B$ and further for each $\mvbeta \in \bar{O}_ {\mvbeta^0} $,  $l(\mvbeta,u)$ has a unique maximizer (in $u$), then the maximizer denoted by $u^*(\cdot)$ is continuous at $\mvbeta^0$.
\end{thm}

We will now consider GERGM model of the form \eqref{gergm} where $H_1$ is a single edge as before and, $H_j$'s are $j$-stars for all $2\leq j\leq s$.  A $j$-star is a undirected graph on $j+1$ vertices with one ``special" vertex, that is, neighbor to the $j$- other vertices and no edge between these $j$ vertices. For this important class of models we will see that the model is solvable, that is, the limiting normalizing constant can be expressed as a scaler optimization problem for all admissible parameter values and we will not need the restriction $\beta_i>0$ for $2\leq i \leq s$. Further the requirement for the base measure to have non-negative support or the even values of the number of edges in $H_i$'s can be relaxed for this class of models. 
\begin{thm}\label{jstar}
Consider the GERGM model with  $T(k) = \sum_{i=1}^s{\beta_i t(H_i, k)}$ where $H_i$'s are $i$-stars for $2\leq i\leq s$ and $\mvbeta$ is admissible.Then the value of the normalizing constant is given by
\begin{equation}\label{stdnormconst:2}
\lim_{n \rightarrow \infty}{\psi_n} = \sup_{u }{\left(\sum_{i=1}^s \beta_i u^{e(H_i)} - I(u)\right)}
\end{equation}
Let $K$ be the set of maximizers of the function $g(\cdot)$ defined via $g(u):= {\sum_{i=1}^s \beta_i u^{e(H_i)} - I(u)} $. Assume that,
\begin{equation}\label{finitemaxermm}
\lim_{|u| \rightarrow \infty}{\sum_{i=1}^s \beta_i u^{e(H_i)} - I(u)} = -\infty.
\end{equation}
 Then $K$ has finitely many elements and almost surely
\begin{equation}\label{hightemp}
\min_{u \in K}\delta(\tilde{k}_n, \tilde{k}^u) \rightarrow 0.
\end{equation}

\end{thm}
As pointed out by an anonymous referee that it would be interesting to study edge-star model under different scaling and one starting point would be the article \cite{rauh2017polytope}. Since all our methods are based on large deviation result with rate $n^2$, to study this model under different scaling one would need more refined large defined result. For example, the result in \cite{chatterjee2016nonlinear} can be used to obtain an evaluation of the normalizing constant when the scaling is different than $n^2$. We will now focus on simple edge-two-star model with standard Normal distribution as the base measure. For convenience we write the model in the following form. For $i< j \in [n]$, let $q_{ij}(=q_{ji})\equiv N(0,1)$ and let the base measure $Q_n$ be the corresponding product measure. The edge-two-star GERGM model is obtained by setting $T$, the exponent of GERGM is given to,
\begin{equation}\label{twostarx}
T(\mvx) = \f{\beta_1}{n^2} \sum_{i\neq j } {x_{ij}} + \f{\beta_2}{n^3}\sum_{i\neq j\neq k \neq i}^n{x_{ij} x_{ik}}.
\end{equation}
Consequently the density with respect to Lebesgue measure is proportional to, $$\exp\left({n^2T(\mvx)- \f{1}{2}\sum_{i<j} {x_{ij}^2}}\right).$$ 
Our next theorem gives explicit value of the normalizing constant when the base measure is standard Gaussian. Our proof technique is to write the second term in \eqref{twostarx} as a suitable quadratic form and then use spectral decomposition. 

\begin{thm}
	\label{thm:normal-full-solvable}
For an edge-two-star model with base measure standard normal distribution the normalizing constant is given by, $\psi_n$, where $\psi_n$ is given by,
\begin{equation}\label{rawnormconst}
\psi_n = \f{1}{\sqrt{1-\f{4\beta_2(n-1)}{n}}} \exp{\left(\f{\beta_1^2 n(n-1)}{1-4\beta_2\f{(n-1)}{n}}\right)} \left( {1-\f{2\beta_2(n-2)}{n}} \right)^{-\f{(n-1)}{2}}.
\end{equation}
whenever $n\geq 3$ and $\beta_2 < \f{n}{4(n-1)}$.
In particular,
\begin{equation}\label{limitnormconst}
\lim_{n \rightarrow \infty}\f{1}{n^2}\ln{\psi_n} = {\left(\f{\beta_1^2 }{1-4\beta_2}\right)}.
\end{equation}

\end{thm}

\begin{rmk}\label{remarkaboutnormal}
 In the proof we will see that the normalizing constant is not finite when $\beta_2 \geq \f{n}{4(n-1)}$, hence it is not a proper probability distribution. Further \eqref{limitnormconst} shows that the limiting normalizing constant is analytic and thus there will be no phase transition in an undirected edge-two-star model if the base measure is standard normal. Since our result is valid for all finite $n \geq 3$, taking partial derivatives of $\log{\psi_n}$ with respect to $\beta_1$ shows that the expected number of edges is a continuous function in $\beta_1$. To the best of our knowledge this is the first proof of ``non-degeneracy'' in the undirected edge-two-star GERGM model that was predicted in \cite{conjec2str} and the proof works even for finite values of $n$. On the other hand this proof heavily relies on the fact that the base measure is Gaussian and the specific class of models, thus it is not immediately extendable to other cases. We have heavily used the fact that orthogonal transformation of standard normal distribution is again a standard normal distribution. Thus for other base measures our technique is not applicable. We would like to thank one of the referee for pointing out the paper \cite{hanneke2009discrete}, where the authors studied non-degeneracy of dynamic (temporal) ERGM. Their result is based on studying the entropy of the dynamic model and their technique and results are not immediately applicable in our case. For other (non-temporal) models one approach would be to apply Theorem \ref{continhightemp} and determine for what values of parameters the model does not suffer from degeneracy. We defer this study for future work but the steps required in concrete examples motivates the next result. 
\end{rmk}

\begin{rmk}
	Regarding specific bases measures, as pointed out by a referee,  another important class of models that should be studied are GERGM models with Poisson distribution as the base measure, as formulated in \cite{krivitsky2012exponential}. We hope to analyze the implications of our results in this specific context in subsequent work. 
\end{rmk}

The following theorem focuses on the edge-two-star model when the base measure is not normal distribution. Consider the density of the measure $q_{ij} = q$ is given by the density (w.r.t. to the Lebesgue measure),
\begin{equation}
\label{eqn:base-measure-four}
	q(x) = C_4 \exp{\left(-x^{4}\right)}, \qquad -\infty < x < \infty,
\end{equation} 
where $C_4 {\int_{\mathbb{R}}\exp{\left(-x^{4}\right)} \, dx} =1$. 
 For fixed $l>0$ recall the truncation operator $f_l(x)$ applied to a number $x$ in \eqref{eqn:trunc}. For simplicity write $x^l$ for this operation on $x$. For a finite vector $\vx$, write $\vx^l$ for the vector obtained by entry-wise truncation operation.
\begin{thm}
	\label{thm:edge-two-star-four}
	Consider the edge-two-star model
	\begin{equation}
	\label{eqn:edge-two-four}
		T(\mvx) = \f{\beta_1}{n^2} \sum_{i\neq j } {x_{ij}} + \f{\beta_2}{n^3}\sum_{i\neq j\neq k \neq i}^n{x_{ij} x_{ik}}
	\end{equation}
	with base measure given by \eqref{eqn:base-measure-four}. Then condition {\bf(C2)$^\prime$} holds namely, 
	\begin{equation}\label{verifycond}
	\limsup_{l \rightarrow \infty} \limsup_{n \rightarrow \infty}\f{1}{n^2}\ln \left(C_4^{n \choose 2}\int_{\{|T(\mvx)-T(\mvx^l)| > \epsilon\}} e^{(n^2T(\mvx)- \sum_{i<j}x_{ij}^4 )} \, dx\right) = - \infty.
	\end{equation}
	 In particular the assertion of Theorem \ref{normconst} and Theorem \ref{jstar} hold for all $(\beta_1,\beta_2) \in \bR\times \bR$.  
\end{thm}

\begin{rmk}
Theorem \ref{thm:edge-two-star-four} tells us that for edge-two-star model where the base measure has density proportional to $e^{-x^4}$, the limiting normalizing constant can always be computed via a scaler optimization problem. In particular the normalizing constant will be equal to $\sup_{u}\left(\beta_1u+\beta_2 u^2 -I(u)\right)$. Here $I(.)$ is the Cramer rate function of the distribution with density given by \eqref{eqn:base-measure-four}. It also gives (using Theorem \ref{jstar}) that a typical observation from this model always concentrate around constant functions where the distance is measured by cut-metric.
\end{rmk}

\section{Graph Homomorphisms and scope of GERGM}
\label{sec:grr-hom-res}
It is now natural to ask about appropriate choices for the function $T$. In the context of applications, the specification of the functionals are domain and application specific.  In the context of the usual exponential random graph model (with unweighted adjacency matrices) commonly used choices are homomorphism counts of certain fixed graphs, for example the number of edges, two-stars, triangles and so on as we have described in some of the above results. In the context of weighted graphs and the general metric space of kernels $(\tilde{\cK},\delta)$ that we work on, we would like to expand on notions like `triangles' or `two-stars' etc and in general understand extensions of such functionals in the context of weighted graphs. Further many models in the context of applications have not just edge weights but also ``node-specific'' co-variates. 
The aim of this section is to make headway on these concepts, in particular understand some extensions of standard functionals in the context of unweighted networks which are still continuous in $(\tilde{\cK},\delta)$, that can incorporate both edge weight information and possible node-level co-variate information; we do not aim for completeness, rather we state one result but the main aim is to show some of the issues involved for defining such objects in the weighted context.   
We present the definitions of of homomorphism density from \cite[Chapter 5]{lngl}. Instead of giving the definition in the most general form we follow a step by step approach similar to \cite[Chapter 5]{lngl}. Recall that our main results in the previous Section required functionals of interest to be continuous on the truncated spaces $\cK^l$. Thus we will mainly deal with functionals defined on these bounded spaces and establish one result for the continuity of functionals (Theorem \ref{contnodeweight}). The remark after this result describes why continuity can fail in the context of general functionals in the weighted context in even simple situations and thus specifications need careful thought.  

We start with the definition of homomorphism in simple unweighted graph. Let $G$ and $H$ be two simple graphs (unweighted) and $V(G)$ and $E(G)$ (respectively $V(H)$ and $E(H)$) be the corresponding  vertex set and edge sets of $G$ (respectively of $H$).  
\begin{dfn}
A function $f: V(G) \rightarrow V(H)$ is called a $\bold{homomorphism}$ if it maps adjacent vertices to adjacent vertices. The set of all such possible homomorphisms is denoted by $\hom(G, H)$ and the cardinality of this set $|\hom(G, H)|$ is called homomorphism number. The ratio $\frac{|\hom(G, H)|}{|V(G)|^{|V(H)|}}$ is called the homomorphism density. 
\end{dfn}
Note that a homomorphism does not necessarily map a non-adjacent pair of vertices to another non-adjacent pair. The homomorphism density represents the probability that a uniform random map $V(G) \rightarrow V(H)$ is a homomorphism. Now let us first give the definition of homomorphism number $V(F) \rightarrow V(G)$ when $G$ is a weighted graph with adjacency matrix $A^G$ and $F$ is a simple graph(for example a triangle). We assign the following weight to every map $\phi : V(F) \rightarrow V(G)$
\[
\hom_{\phi}(F, G) = \prod_{(i,j) \in E(F)}{A^G_{\phi(i), \phi(j)}}.
\]
Now we define the homomorphism number $|\hom(F, G)|$ as
\[
|\hom(F, G)| = \sum_{\phi : V(F) \rightarrow V(G)}{\hom_{\phi}(F, G)}.
\]
Note that in the above display $|.|$ does not represent cardinality unless both $F$ and $G$ are un-weighted graphs. 

Now we turn to the  case when $F$ and $G$ both are weighted graphs. This quantity cannot be defined for arbitrary weights and we need some restrictions on the weights. For example if any one of the following two conditions hold then $|\hom(F, G)|$ is well defined:
\begin{enumeratea}
	\item the edge-weights of $F$ are non-negative integers and with the convention $0^0 =1$ (here we do not need any restriction on the weights of $G$), for example in Theorem \ref{thm:normal-full-solvable} the base measure is standard normal (the edge-weights of $G$ can be any real number) and in edge-two-star model the edge-weights are indeed non-negative integers;
	\item the edge-weights of $G$ are positive (in this case we do not need any restriction on the weights of $F$).
\end{enumeratea}

 In general $|\hom(F, G)|$ is defined when the weights in \eqref{condreq} below are well defined. Now let us incorporate node-level covariates. We start with the general framework and then give some specific examples for illustration. Suppose associated with a graph $G$ one has node-weights $\set{\alpha_i(G): i\in [n]}$, where $\alpha_i(G)$  denotes the weight of node $i$ in the graph $G$. 
 For every map $\phi: V(F) \rightarrow V(G)$,  define weights
\begin{equation}\label{condreq}
\alpha_\phi = \prod_{i \in V(F)}{\alpha_{\phi(i)}(G)^{\alpha_{i}(F)}},
\end{equation}
and
\[
\hom_{\phi}(F, G) = {\prod_{(i,j) \in E(F)}{(A^G_{\phi(i), \phi(j)})^{A^F_{i,j}}}  }.
\] 
Let $\bar{\mvalpha}(F,G) = \{\alpha_\phi,  \phi: V(F) \rightarrow V(G)\}$. Define the homomorphism number $|\hom(F, \bar{\alpha}, G)|$ to be 
\begin{equation}\label{homnum}
|\hom(F, \bar{\alpha}, G)| = \sum_{\phi : V(F) \rightarrow V(G)}{\alpha_\phi \hom_{\phi}(F, G)}.
\end{equation}
Note that equation \eqref{homnum} extends the notion of homomorphism to weighted graphs. In particular, if we assume that both $G$ and $F$ are unweighted then \eqref{homnum} is the usual homomorphism number. Since the graphs are now weighted $\alpha_\phi$ represents the weight of the map $\phi$. For unweighted graphs $\alpha_\phi$ can take only two values, zero and one. We refer the reader to \cite[Chapter 5]{lngl} for detailed discussion. Finally we define the homomorphism density $t(F, \bar{\alpha}, G)$ for any weighted graph $F$ and $G$, 
\begin{equation}\label{homden}
t(F, \bar{\alpha}, G) = \frac{|\hom(F, \bar{\alpha}, G)|}{|V(G)|^{|V(H)|}}.
\end{equation}
In \eqref{homden}, as before $|V(G)|$ and $|V(H)|$ denotes the number of nodes in $G$ and $H$. We already know when $F$ is un-weighted(all node-weights and edge weights equal to one) then $t(F, G)$ is continuous in cut-metric  as long as the $G$ has no  node-weight and edge weights of $G$ are bounded. We will generalize this to incorporate node-level covariates. We provide two concrete examples below.

\begin{ex}[Weighted triangle counts]
	Suppose $F$ is a triangle and $G$ be a complete graph on $n$ vertices, then $A_{ij}^G =1$ for all $i\neq j$ and $A_{ii}^G =0$. Thus for a map $\phi: V(F) \rightarrow V(G)$, $\hom_{\phi}(F, G) = \prod_{(i,j) \in E(F)}{A^G_{\phi(i), \phi(j)}} =1$ whenever $\phi$ maps the vertices of $F$ to three distinct vertices in $G$ and equal to zero otherwise. Thus the homomorphism number is equal to $n(n-1)(n-2)$. On the other hand if the node weight of the $i$-th node of graph $G$ is $\alpha_i(G)$, and $F$ is a triangle then the homomorphism number is equal to \[|\hom(F,\bar{\alpha}, G)| = \sum_{1\leq i\neq j \neq k \leq n}{\alpha_i(G) \alpha_j(G) \alpha_k(G)}.\] 
\end{ex}

\begin{ex}[Node-level covariates]
	Suppose  $F$ is just a single vertex with no loop and weight $1$ and $G$ is a graph on $n$ vertices with node weights $\alpha_i(G)$, then with the convention that empty product is one, we will have \[|\hom(F,\bar{\alpha}, G)| = \sum_{i \in [n]} \alpha_i(G).\]
\end{ex}
Now we can generalize the above definitions to a kernel $k \in \mathcal{K}$ as follows. Assume that $\alpha:[0,1] \rightarrow [0, \infty)$ is the ``node-weight" function of  a kernel $k$. Fix (an edge and node) weighted graph,
 \begin{equation}
 \label{eqn:weight-graph-F}
 	F = (V(F), E(F), (A_{ij}^F)_{i,j\in V(F)}, (\alpha_i(F))_{i\in V(F)}).
 \end{equation}  
 Define,
\begin{equation}\label{homdenker}
t(F, \alpha, k) = \int_{[0, 1]^{|V(F)|}}  \prod_{i \in V(F)}{\alpha(x_i)^{\alpha_{i}(F)}}   {\prod_{(i,j) \in E(F)}k(x_i, x_j)^{A^F_{ij}}} \, \prod_{i \in V(F)}{dx_i}
\end{equation}
Recall that a fixed weighted graph $G$ can be mapped to a kernel $k_G$ by \eqref{mapintokernel} and recall the partition of the unit interval $\set{J_i^n:1\leq i\leq n}$ where $J_1^n =[0, \f{1}{n}]$ and $J_i^n = (\frac{i-1}{n}, \frac{i}{n}]$ for $i=2,\ldots,n$. Now for a general node and edge weighted graph $G$ on $n$ vertices:
\[G = (V(G), E(G), (A_{ij}^G)_{i,j\in V(G)}, (\alpha_i(G))_{i\in V(G)}),\]
if we 
 define the function
  $$\alpha_G(x) = \alpha_i(G)^{\alpha_i(F)}\bold{1}_{J_i^n}(x), \qquad 0\le x\le 1.$$ 
     Then it is easy to check that,
\[
t(F, \bar{\alpha}, G) = t(F, \alpha_G, k_G).
\]
By \cite[Theorem 3.7(a)]{gl2} it follows  that \ref{homdenker} is continuous in cut-metric as long as $F$ is an un-weighted graph. We next generalize this result in the case when $F$ has bounded node-weights.
\begin{thm}\label{contnodeweight}
Suppose $F$ is a node-weighted simple graph (edge-weights $A_{ij}^F \equiv 1$).  Let $\set{k_n}_{n\geq 1}$ be a sequence of kernels absolutely bounded by a constant $M$ and converging to $k$ in the cut-metric. For each $n$ assume that kernel $k_n$ has node-weight function $\alpha_n$ and let $\alpha$ be the node weight function of $k$. Assume that $\alpha_n \rightarrow \alpha$ uniformly on $[0, 1]$ as $n\to\infty$. Then, 
\begin{equation}
t(F, \alpha_n, k_n) \rightarrow t(F, \alpha, k) 
\end{equation}
\end{thm}

\begin{rmk}\label{counterexampleforedgeweight}
It is natural to ask what happens if $F$ is an edge-weighted graph. We show by an example that continuity breaks down if one is not careful and thus this topic needs further investigation. Consider a sequence of simple graph $\{G_n\}$ with node-weights equal to one, such that $k_{G_n} \rightarrow k_G$ in cut-metric, where $k_G(x,y) = p$ for all $(x,y) \in D$ for some $0<p<1$. Now let $F$ be a single edge(no-node-weight) with edge-weight $2$ and $F'$ be a single edge (no-node weight) with edge weight $1$. By \cite[Theorem 3.7(a)]{gl2} or Theorem \ref{contnodeweight}, we have $t(F', G_n) \rightarrow t(F',k_G) =p$ as $n\to\infty$. Again, since $G_n$ has no edge weight we have $t(F, G_n) = t(F', G_n) \rightarrow t(F',k_G) = p < p^2 = t(F, k_G)$. This shows that if $F$ is an edge-weighted graph the continuity fails to hold. Note that in this paper we do not distinguish between an edge with weight zero and a missing edge. As kindly pointed out by one of the anonymous referee that this results in the lack of continuity in the aforementioned example. It would be interesting to consider the edge-weights and adjacency matrix as different structures and we defer this study to our future work.
\end{rmk}

\begin{rmk}
	In applications (see \cite{desmarais2012micro}, also see \cite{demuse2017mixing} for mixing time analysis of vertex weighted exponential random graphs) of GERGM one often constructs edge level statistics from the values associated with a node level covariate. The simplest example would be to consider the sum of the values of two nodes, i.e. for a graph $G$ with vertex set $[n]$ associate $\alpha_i(G) + \alpha_j(G)$ to the pair $(i,j)$ for $i,j \in [n]$. Then the statistics becomes $\f{1}{n^2}\sum_{i,j =1}^n(\alpha_i(G) + \alpha_j(G)) = \f{2}{n}\sum_{i=1}^n\alpha_i(G)$. If we associate $\alpha_i(G)\alpha_j(G)$ to the pair $(i,j)$ for $i,j \in [n]$, then the statistics becomes $\f{1}{n^2}\sum_{i\neq j =1}^n\alpha_i(G)\alpha_j(G)$ and this is homomorphism density of an edge in a node-weighted simple graph $G$. One can construct such statistics involving any number of nodes (for example triangle count involves sum of product of all possible triplets of nodes). Theorem \ref{contnodeweight} shows under suitable assumption on the node weights $\alpha_i(G)$ and the graph sequence $G$ the statistics discussed in the last section are continuous in cut-metric. 
	
\end{rmk}

\section{Discussion}
\label{sec:disc}
In this Section we discuss related work as well as suggest some open problems that would have impact in applications. 

	\subsection{Related work:}  Weighted exponential random graph models were theoretically analyzed in \cite{yin2016phase} when the base measure is supported on a bounded interval and in \cite{demuse2017phase} the authors analyzed the phase transition phenomenon for a class of base measures supported on $[0,1]$. In \cite{yin2016phase} the ``no-phase transition" result for standard normal base measure was proved for directed edge-two-star model. Motivated by applications \cite{krivitsky2012exponential} we extend this work when the base measure is supported on the whole real line. We showed for general base measure the model does not suffer degeneracy in ``high-temperature" regime. Also, via an explicit calculation we have showed for standard normal distribution the  undirected edge-two-star model does not admit a phase transition. Finally under certain assumptions we established continuity of homomorphism densities of node-weighted graphs in cut-metric. We have only begun an analysis of this model and for the sake of concreteness, after the general setting of the main result, explore the ramifications for a few base measures. Other examples of bases measures of relevance from applications including count data can be found in \cite{krivitsky2012exponential}. It would be interesting to explore these specific models and rigorously understand degeneracy (or lack thereof) for various specifications motivated by domain applications. 
	
     \subsection{Relevance of this work and open problems:} The GERGM model has stimulated a number of research directions, both in the context of rigorous theory as described above and methodological aimed at efficient simulation under a host of model specifications \cite{conjec2str}. We will now propose a number of open directions motivated by some of the results in this paper. We will occasionally eschew rigor in order to give an understandable overview of the proposed open direction.  We start by briefly describing how these models are used in practice defering a full description to \cite{desmarais2012statistical,conjec2str}.  Given data $\vX_n$ assumed from some distribution $R_n(\cdot, \mvbeta^0)$ as in \eqref{ermmdfn} with specification $T(\cdot)$ of the form, 
\begin{equation}
\label{eqn:tvx-general}
	T(\vx):= \sum_{i=1}^s \beta_i^0 T_i(\vx), \qquad \vx\in \tilde{\cK},
\end{equation}
	where $s$ is fixed,  $T_i(\vx)$ are domain specific continuous functions (see e.g. \eqref{gergm}) and $\mvbeta^0 := (\beta_1^0, \ldots, \beta_s^0)$ are the driving parameters of the models. Here we assume unknown parameter $\mvbeta^0$ and the superscript ``$0$'' is to indicate the ``true'' parameter.  The aim then is to estimate $\mvbeta^0$ from the single observation $\vX_n\sim R_n(\cdot,\mvbeta^0)$. One of the main techniques used is maximum likelihood which in the context of the model in \eqref{ermmdfn} is to use the estimator,  
	\begin{equation}
	\label{eqn:mle-form}
		\hat{\mvbeta}:= \arg \max_{\mvbeta\in {\bR^s}} \left(\sum_{i=1}^s \beta_i T_i(\vX_n) -  \psi_n(\mvbeta)\right).
	\end{equation}
This explains the importance or goal of deriving results such as Theorem \ref{normconst}. In the context of applications, for fixed $n\geq 1$, the starting point in the above optimization problem is getting a handle on $\psi_n(\mvbeta)$ numerically, via MCMC techniques such as Gibbs sampling or Metropolis-Hasting. This suggests the first set of problems.

\begin{open}[Mixing time of MCMC algorithms]
	\label{op:mcmc}
	Fix $\mvbeta$ and a collection of statistics as in \eqref{eqn:tvx-general}. Consider either the Gibbs sampler or the Metropolis-Hastings algorithm for simulating $R_n(\cdot, \mvbeta)$. Establish conditions for quick mixing of these chains. In particular consider the model in \eqref{gergm} in the setting of Theorem \ref{thm:spec-gergm}. Then we conjecture that under general conditions the following behavior should hold:
	\begin{enumeratea}
		\item If \eqref{stdnormconst} has a unique maximizer then the corresponding samplers mix quickly (polynomial in the size of the network $n$). 
		\item If \eqref{stdnormconst} has multiple distinct maximizers then the corresponding samplers take exponentially long to mix. 
	\end{enumeratea}
\end{open}

We direct the interested reader to \cite{levin2009markov,meyn2012markov,robert-casella} for more details on convergence methodology and \cite{shankar} for related results in the context of ERGMs. In the context of the unweighted setting, \cite{cha} derived a number of fundamental results for the model, in particular showing a number of deficiencies  of this model under various model specifications, in particular showing that in various regimes, the model is close to a standard Erdos-Renyi random graph with independent edge probability (especially in the ferro-magnetic regime where the parameters $\beta$ are assumed positive) in the large network $n\to\infty$ limit. Despite this, for finite $n$, practitioners have found that the limits for the normalizing constants derived in \cite{cha} perform surprisingly well for estimation of these models even for small $n$ (e.g. $n=20$). Further, in order to fix issues with the standard ERGM, a number of practitioners have proposed fixes such as the ``alternating signs''  model in \cite{snijders2006new}. Mathematical theory for the improved performance of this model is still lacking (albeit a specific case was studied in \cite{cha}).   It would be interesting to see if (a) the results in this paper allow similar simplifications  of the normalizing constant  for estimation in the weighted context and (b) if the models suffer similar issues as in the unweighted case, if fixes such as the ``alternating signs'' approach work in this context.

The above discussion suggests the following question. 

\begin{open}[Other base measures]
	This paper considered two particular base measures, the normal distribution and density with $e^{-x^4}$ tail. Consider other base measures required in applications. Derive asymptotics for degeneracy or lack thereof for these models. 
\end{open}
Interested readers can find more examples of base measures especially in the unbounded regime in \cite{krivitsky2012exponential}.
 Continuing with the description of the estimation problem \eqref{eqn:mle-form}, there are two main steps: 

\begin{enumeratea}
	\item {\bf Initialization:} Find an initial candidate vertex $\hat{\mvbeta}_0$. Techniques include objects such as maximizing pseudo-likelihood. 
	\item {\bf Maximization and MCMC:} First note that, writing $\E_{\mvbeta}$ for expectation with respect to $R_n(\cdot, \mvbeta)$ for any arbitrary $\mvbeta, \mvbeta^\prime \in \bR^s$, 
	\[\exp(n^2 (\psi_n(\mvbeta^\prime) - \psi_n(\mvbeta))) =  \E_{\mvbeta}\left(\exp(n^2(\mvbeta^\prime - \mvbeta )T(\vX))\right)\]
\end{enumeratea}
Thus if we are able to simulate from the distribution with parameter $\mvbeta$, an estimate of the functional $\psi_n(\mvbeta)$ can be obtained via $M$ samples $(\vX^{\sss(1)}_n, \ldots, \vX_n^{\sss(M)})$ from $R_n(\cdot, \mvbeta)$ via the proxy estimate,
\[\hat{\psi}_{n}(\mvbeta^\prime;\mvbeta):= \frac{1}{M}\sum_{i=1}^M \left(\exp(n^2(\mvbeta^\prime - \mvbeta )T(\vX_n^{\sss(i)}))\right) \]
This suggests the following iterative scheme: Let $\mvbeta_r$ be the current estimator. Obtain $\mvbeta_{r+1}$ via 
\begin{equation}
\label{eqn:mle-quasi}
	\mvbeta_{r+1} = \arg \max_{\mvbeta\in \bR^s} \left(\sum_{i=1}^s \beta_i T_i(\vX_n) -  \hat{\psi}_{n}(\mvbeta;\mvbeta_r)\right).
\end{equation}
Various stopping mechanisms for the above iterative scheme are then implemented in practice. Methodology and numerical techniques  to carry out the above scheme specific to GERGM specifications were formulated in \cite{conjec2str}. The math results in this paper suggest the following.

\begin{open}[Insight for MLE]
	Using the results in this paper, obtain insight into the consistency (or lack thereof) of the above MLE. In particular: 
	\begin{enumeratea}
		\item Initialization of the above algorithm is a major issue and has an enormous impact on the running time of the algorithm (even on networks of size $n=50$ say). In light of Theorem \ref{thm:spec-gergm}, explore conditions under which the pseudo-maximum likelihood estimator a ``good'' starting point in practice. 
		\item Are there regimes under which, one can numerically solve the optimization problem in \eqref{stdnormconst} $\psi_n(\mvbeta)$ and uses these in place of $\psi_n(\mvbeta;\mvbeta_r)$ in \eqref{eqn:mle-quasi}? How do these techniques work in practice? 
	\end{enumeratea}

\end{open}
The next problem we believe is hard but important and we were reminded by a referee.  

\begin{open}[Extensions of these models to the sparse regime]
	All the theory in this paper is built up using (extensions) of the theory of dense graph limits \cite{gl1,gl2,gl3} in the context of the weighted regime. Whilst the current set of applications of this methodology have also largely been in the context of networks (such as migration flow networks in \cite{desmarais2012statistical}) where a non-trivial density of edges were non-zero, applications of this methodology for large networks will inevitably lead to the development of this theory for \emph{sparse} networks where $o(n^{2})$ of the edges are non-zero. Fundamental theory for the sparse regime even in the context of the unweighted binary setting has proven to be extremely challenging \cite{BCCY14,BCCY-2-14}. Can any of this be extended to the weighted context? More importantly, are there settings where the weighted setting makes it \emph{easier} to develop theory for specific base measures?
\end{open}

Finally Section \ref{sec:grr-hom-res} suggests the following vein of research. 
\begin{open}[Continuity of functionals]
	A wide array of specifications including node-specific covariates as well as base measures have been proposed in practice \cite{desmarais2012statistical,krivitsky2012exponential}. Explore and develop general conditions for continuity of these functionals in the context of the space $(\tilde{\cK},\delta)$ defined in Section \ref{sec:graph-lim-ldp} and apply Theorem \ref{normconst} to get precise evaluation of the limiting normalizing constants.  
\end{open}

\section{Proofs}
\label{sec:proof}
This section contains proofs of all our results. 
\subsection{Proof of Theorem \ref{normconst}:}

We start with the following elementary Lemma on the role of truncation on the rate functions on our setup. 

\begin{lem}\label{lemteclem1}
	Fix probability measure $\mu$ on $(\bR, \cB(\bR))$ and assume that $\int e^{\theta u} \, d\mu(u) < \infty$ for all $\theta \in \bR$. For fixed $l> 0$, define,
	\[h_l(x) = \sup_{\theta}\left[ \theta x - \ln \int e^{\theta f_l(u)} \, d\mu(u) \right],\]
 and 
 \[h(x) = \sup_{\theta}\left[ \theta x - \ln \int e^{\theta u} \, d\mu(u) \right].\] 
 Then given any $\eps$, $\exists $ $L(\epsilon) <\infty$ such that for $l> L(\epsilon)$ and all $x\in \bR$, we have,
\begin{equation}\label{teclem1}
h(x) \leq h_l(x) + \epsilon,
\end{equation}

\end{lem}

\noindent {\bf Proof:}

First note that 
\begin{equation}\label{teclem1eq0}
h(x) = \max\left(\sup_{\{\theta > 0\}}\left[ \theta x - \ln \int e^{\theta u} \, d\mu(u) \right], \sup_{\{\theta < 0\}}\left[ \theta x - \ln \int e^{\theta u} \, d\mu(u) \right] \right).
\end{equation}
Further we have for each fixed $\theta$, 
\begin{equation}\label{teclem1eq1}
\theta x \le  \sup_{\theta^\prime}\left[ \theta^\prime x - \ln \int e^{\theta^\prime f_l(u)} \, d\mu(u) \right] + \ln \int e^{\theta f_l(u)} \, d\mu(u),
\end{equation}

The proof consists of the following: 

\emph{Step 1} For any given $\epsilon>0$ we will show there exists $L^{+}(\epsilon) >0$ such that for all $l> L^{+}(\epsilon)$ and all $x\in \bR$,   $\sup_{\theta >0} \left[ \theta x  -\ln\left(\int e^{\theta u } \, d\mu(u)\right)\right]  \leq h_l(x)  + \epsilon$.  \\ 

\emph{Step 2} For any given $\epsilon>0$ we show there exists  $L^{-}(\epsilon) >0$ such that for all $l> L^{-}(\epsilon)$ and $x\in \bR$,   $\sup_{\theta <0} \left[ \theta x  -\ln\left(\int e^{\theta u } \, d\mu(u)\right)\right]  \leq h_l(x)  + \epsilon$. \\

Note that \emph{Step 1}, \emph{Step 2} and \eqref{teclem1eq0} immediately imply \eqref{teclem1} with $L = \max(L^{+}(\epsilon), L^{-}(\epsilon))$. We will now prove \emph{Step 1} only as the proof of \emph{Step 2} is identical.

\emph{Proof of step 1}
We use definition of $f_l$  and use $\theta>0$ to write the following,
\begin{align}\label{teclemeqn2}
\theta x &\leq h_l(x) + \ln\left[\int_{-\infty}^{-l} e^{-\theta l } \, d\mu(u) + \int_{-l}^{l} e^{\theta u } \, d\mu(u) + \int_{l}^{\infty} e^{\theta l } \, d\mu(u)\right] \notag \\
& \leq  h_l(x) + \ln\left[\int_{-\infty}^{-l}  \, d\mu(u) + \int_{-l}^{l} e^{\theta u } \, d\mu(u) + \int_{l}^{\infty} e^{\theta u } \, d\mu(u)\right] \notag  \\
&= h_l(x) + \ln\left[\mu( -\infty, -l) + \int_{-l}^{\infty} e^{\theta u } \, d\mu(u) \right].
\end{align}
At this end we observe that the quantity inside logarithm in the last display converges uniformly on $\theta >0$ as $l \rightarrow \infty$, to see this,
\begin{align*}
|\mu( -\infty, -l) + \int_{-l}^{\infty} e^{\theta u } \, d\mu(u) -\int_{-\infty}^{\infty} e^{\theta u } \, d\mu(u)| &< \mu( -\infty, -l) + \int_{-\infty}^{-l} e^{\theta u } \, d\mu(u) \\
&< 2 \mu( -\infty, -l).
\end{align*}

Now using the fact that the quantity inside logarithm in \eqref{teclemeqn2} converges uniformly on $\{\theta>0\}$ to $\int e^{\theta u } \, d\mu(u)$ and the limit is strictly positive (hence log is continuous), we can have a positive number $L^{+}(\epsilon)$ such that for all $l >L^{+}(\epsilon)$,
\[
\theta x \leq h_l(x) + \ln\left[\int e^{\theta u } \, d\mu(u)\right] + \epsilon,
\]
for all $\theta >0$, yielding,
\[
\sup_{\theta >0} \left[ \theta x  -\ln\left(\int e^{\theta u } \, d\mu(u)\right)\right]  \leq h_l(x)  + \epsilon.
\]
The proof is complete.

\qed

\noindent {\bf Proof of Theorem \ref{normconst}: }

	{\bf Upper bound:} We will first show
	\begin{equation}\label{eqn:upper-bound}
	\limsup_{n\rightarrow \infty}{\psi_n} \leq 3\epsilon + \liminf_{l \rightarrow \infty} \sup_{\tilde{h} \in \tilde{\mathcal{K}}^l}(T(\tilde{h})-I(\tilde{h}))
	\end{equation}
	Using \eqref{ermm}, for any $l >0$ and $\epsilon>0$, we decompose as follows:
\begin{align}\label{thm1dis1}
\exp(n^2 \psi_n) = \int_{ \{T(\tilde{k}) - T(f_l(\tilde{k})) \geq \epsilon\}} e^{n^2T(\tilde{k})} \, d\tilde{Q}_n(\tilde{k}) + \int_{ \{T(\tilde{k}) - T(f_l(\tilde{k})) < \epsilon\}} e^{n^2T(\tilde{k})} \, d\tilde{Q}_n(\tilde{k})
\end{align}
For the first term in (\ref{thm1dis1}), using \eqref{assneglect} we have,
\begin{equation}\label{bigneglect}
\limsup_{l \rightarrow \infty}\limsup_{n \rightarrow \infty}\f{1}{n^2} \ln \int_{ \{T(\tilde{k}) - T(f_l(\tilde{k})) \geq \epsilon\}} e^{n^2T(\tilde{k})} \, d\tilde{Q}_n(\tilde{k}) \le  -\epsilon',
\end{equation}for some $\epsilon' >0$, alternatively using \eqref{assneglect1} we have,

\begin{equation}\label{bigneglect1}
\limsup_{l \rightarrow \infty}\limsup_{n \rightarrow \infty}\f{1}{n^2} \ln \int_{\{T(\tilde{k}) - T(f_l(\tilde{k})) \geq \epsilon\}} e^{n^2T(\tilde{k})} \, d\tilde{Q}_n(\tilde{k}) = -\infty.
\end{equation}

For the second term in (\ref{thm1dis1}) by change of variable formula we get the following,
\begin{align}\label{eqn:rightcorrectstep1}
\int_{ \{T(\tilde{k}) - T(f_l(\tilde{k})) < \epsilon\}} e^{n^2T(\tilde{k})} \, d\tilde{Q}_n(\tilde{k})  &\leq \int_{ \{T(\tilde{k}) - T(f_l(\tilde{k})) < \epsilon\} } e^{n^2{(T(f_l(\tilde{k})) + \epsilon )}} \, d\tilde{Q}_n(\tilde{k}), \notag \\
& \leq \int_{ \mathcal{\tilde{K}} } e^{n^2{(T(f_l(\tilde{k})) + \epsilon )}} \, d\tilde{Q}_n(\tilde{k}),  \notag \\
&\leq   \int_{ \tilde{\mathcal{K}}^l} e^{n^2({T(\tilde{k}) + \epsilon)}} \, d\tilde{Q}_n f_l^{-1}(\tilde{{k}}). 
\end{align}

{
Now note that $T:\tilde{\cK}^l\to \bR$ is a bounded continuous function. This implies for fixed $\epsilon >0$, we can obtain a finite set $r_l\subseteq \bR$ such that the intervals $\{(a, a+\epsilon): a \in r_l\}$ covers the range of $T$ restricted to $\mathcal{\tilde{K}}^l$. Further for $a\in r_l$, the set $\tilde{C}_a^l:= (T)^{-1}[a,a+\epsilon] \cap \mathcal{\tilde{K}}^l$ is a closed set in $\mathcal{\tilde{K}}^l$. Since $\cup_{a \in r_l}\tilde{C}_a^l$ covers $\mathcal{\tilde{K}}^l$ by construction, we have the following from \eqref{eqn:rightcorrectstep1},
}

\begin{equation*}
 \int_{k \in \tilde{\mathcal{K}^l}} e^{n^2({T(\tilde{k}) + \epsilon)}} \, d\tilde{Q}_nf_l^{-1}(\tilde{{k}})  \leq \sum_{a \in r_l} \int_{  \tilde{C}_a^l } e^{n^2(T(\tilde{k}) + \epsilon)} \, d\tilde{Q}_nf_l^{-1}(\tilde{{k}}) 
\end{equation*}
Now in the right hand side of the last display, if $\tilde{k} \in \tilde{C}_a^l$ then $T(\tilde{k}) \leq a+\epsilon$, yielding,
\begin{align*}
 \int_{k \in \tilde{\mathcal{K}^l}} e^{n^2({T(\tilde{k}) + \epsilon)}} \, d\tilde{Q}_nf_l^{-1}(\tilde{{k}})  &\leq \sum_{a \in r_l} \int_{  \tilde{C}_a^l } e^{n^2(T(\tilde{k}) + \epsilon)} \, d\tilde{Q}_nf_l^{-1}(\tilde{{k}}),  \\
& \leq |r_l| \sup_{a \in r_l} {e^{n^2(a + 2\epsilon)}\tilde{Q}_n f_l^{-1}(\tilde{C}_a^l)}. 
\end{align*}
Since the support of $\tilde{Q}_n f_l^{-1}$ is bounded,  using the upper bound \eqref{ldpclosed} in Theorem \ref{thm:cha-varadhan} results in the following estimate,
\begin{align*}
&\limsup_{n \rightarrow \infty}{\frac{1}{n^2} \ln{\int_{k \in \tilde{\mathcal{K}} \cap \{T(\tilde{k}) - T(f_l(\tilde{k})) < \epsilon\}} e^{n^2T(\tilde{k})} \, d\tilde{Q}_n(\tilde{k})}}  \\ 
&\leq \sup_{a \in r_l}(a + 2\epsilon + \limsup_{n \rightarrow \infty}{\frac{1}{n^2}{\ln{\tilde{Q}_nf_l^{-1}(\tilde{C}_a^l )}}} ) \\
&  \leq \sup_{a \in r_l}(a + 2\epsilon - \inf_{\tilde{k} \in \tilde{C}_a^l }{I_l(\tilde{k})}) \\
& \leq \sup_{a \in r_l}(a + 2\epsilon - \inf_{\tilde{k} \in \tilde{C}_a^l }{I_l(\tilde{k})}).
\end{align*}
Here, 
\begin{equation}\label{ratefunction1}
I_l(k) = \frac{1}{2} \iint_{D}{h_l(k(x,y))} \, dx \,dy,
\end{equation}
where
\begin{equation}\label{con2A2}
h_l(x) = \sup_{\theta}{[\theta x - \ln M_l(\theta)]}
\end{equation}
 and $M_l(\theta) = \int{e^{\theta f_l(x)}} \,\mu(dx)$.
Now for each $\tilde{k} \in \tilde{C_a^l}$ we have $T(\tilde{k})\geq a$, hence we have
\[
\sup_{\tilde{k} \in \tilde{C_a^l}}(T(\tilde{k})-I_l(\tilde{k})) \geq \sup_{\tilde{k} \in \tilde{C_a^l}}(a- I_l(\tilde{k})) = a - \inf_{\tilde{k} \in \tilde{C_a^l}}{I_l(\tilde{k})}.
\]
Thus we have,
\begin{align}
\limsup_{n \rightarrow \infty} \f{1}{n^2} \ln\int_{ \{T(\tilde{k})) - T(f_l(\tilde{k})) < \epsilon\}} e^{n^2T(\tilde{k})} \, d\tilde{Q}_n(\tilde{k})
 & \leq \sup_{a \in r_l}(a + 2\epsilon - \inf_{\tilde{k} \in \tilde{C}_a^l}{I_l(\tilde{k})}) \notag \\
& \leq 2\epsilon + \sup_{a \in r_l}  \sup_{\tilde{k} \in \tilde{C}_a^l}(T(\tilde{k})-I_l(\tilde{k}))  \notag \\ 
& = 2\epsilon + \sup_{\tilde{k} \in \tilde{\mathcal{K}}^l}(T(\tilde{k})-I_l(\tilde{k})) \label{finiteapprox} 
\end{align}
From Lemma \ref{lemteclem1}, given $\epsilon$ as above, we may choose $L(\epsilon)$ such that for $l > L(\epsilon)$, 
\[
I_l(k) \geq I(k) -\epsilon.
\]
From Lemma \ref{lemteclem1} it follows that the last display holds for any given $\epsilon$ as long as $l$ is large. Thus we have for $l > L(\epsilon)$,
\begin{align}
\limsup_{n \rightarrow \infty} \f{1}{n^2} \ln\int_{ \{T(\tilde{k})) - T(f_l(\tilde{k})) < \epsilon\}} e^{n^2T(\tilde{k})} \, d\tilde{Q}_n(\tilde{k})
&\leq 3\epsilon + \sup_{\tilde{k} \in \tilde{\mathcal{K}}^l}(T(\tilde{k})-I(\tilde{k})) \label{nonneglect} 
\end{align}

To complete the proof of the upper bound, we use estimates on \eqref{bigneglect} (under {\bf (C2)}) or \eqref{bigneglect1} (under {\bf C2}$^\prime$) to show that the second term in \eqref{thm1dis1} does not contribute at $n^2$ scale. To show this we will use the fact $\limsup{\frac{1}{n^2}\ln(a_n + b_n)} = \max(\limsup{\frac{1}{n^2}\ln(a_n)}, {\limsup{\frac{1}{n^2}\ln(b_n)}})$. Combining \eqref{nonneglect} and \eqref{bigneglect} and letting $l\to\infty$ we get,
\begin{equation*}
\limsup_{n\rightarrow \infty}{\psi_n} \leq \max\left( 3\epsilon + \liminf_{l\rightarrow \infty}\sup_{\tilde{h} \in \tilde{\mathcal{K}^l}}(T(\tilde{h})-I(\tilde{h})),   -\epsilon'\right)
\end{equation*}
for some $\epsilon'>0$. By our assumption \ref{assneglect} the second term in the above maximum is strictly smaller than the first one and we have,
\begin{equation}\label{rightside}
\limsup_{n\rightarrow \infty}{\psi_n} \leq 3\epsilon + \liminf_{l\rightarrow \infty}\sup_{\tilde{h} \in \tilde{\mathcal{K}^l}}(T(\tilde{h})-I(\tilde{h}))
\end{equation}

Alternatively using \eqref{nonneglect} and \eqref{bigneglect1} we have,
\begin{align*}
&\limsup_{n\rightarrow \infty}{\psi_n} \\ 
&\leq \max\left(\limsup_{n \rightarrow \infty} \f{1}{n^2} \ln\int_{ \{T(\tilde{k}) - T(f_l(\tilde{k})) < \epsilon\}} e^{n^2T(\tilde{k})} \, d\tilde{Q}_n(\tilde{k}), \limsup_{n \rightarrow \infty}\f{1}{n^2} \ln \int_{ \{T(\tilde{k}) - T(f_l(\tilde{k})) \geq \epsilon\}} e^{n^2T(\tilde{k})} \, d\tilde{Q}_n(\tilde{k})
    \right)
\end{align*}
Letting  $l\to\infty$ implies the second term goes to $-\infty$ and will again yield \eqref{rightside}. Since $\epsilon$ was arbitrary, this completes the proof of the upper bound \eqref{eqn:upper-bound}. 

{\bf Lower bound:} We will now show:

\begin{equation}\label{eqn:lowerbound}
\liminf_{n\to \infty}{\psi_n} \geq  \sup_{\tilde{h} \in \tilde{\mathcal{K}^l}}(T(\tilde{h})-I(\tilde{h}))
\end{equation}
for each $l>0$.
By Continuity of $T$ on $\tilde{\mathcal{K}^l}$,  $\tilde{U}_a^l := (T)^{-1}(a,a+\epsilon) \cap \tilde{\mathcal{K}^l} $ is an open set and by construction $\cup_{a \in r_l} \tilde{U}_a^l$ covers {$\tilde{\mathcal{K}}^l$}.  Thus we have,
\begin{align*}
\exp(n^2 \psi_n) = & \int_{k \in \tilde{\mathcal{K}}} e^{n^2T(\tilde{k})} \, d\tilde{Q}_n(\tilde{k}) \\
&\geq \int_{k \in \tilde{\mathcal{K}}^l} e^{n^2T(\tilde{k})} \, d\tilde{Q}_n(\tilde{k}) \\
&\geq  \int_{{U_a^l}  }e^{n^2a} \, d\tilde{Q}_n(\tilde{k}) \\
& = e^{n^2a} \tilde{Q}_n(U_a^l).
\end{align*}

The third line follows from the fact that if $\tilde{k} \in U_a^l $ then $T(\tilde{k})>a$. 
Thus we have, 
\begin{align*}
\liminf_{n \rightarrow \infty} \psi_n &\geq a +  \liminf_{n \rightarrow \infty} \f{1}{n^2} \ln{\tilde{Q}_n(\tilde{U}_a^l )} 
\end{align*}
Hence by \ref{ldpopen} we have
\begin{align*}
\liminf_{n \rightarrow \infty}{\psi_n} \geq a - \inf_{\tilde{k} \in \tilde{U_a^l}}{I(\tilde{k})}
\end{align*}
Now for each $\tilde{k} \in \tilde{U_a^l}$ we have $T(\tilde{k})< a + \epsilon$. Thus
\[
 \sup_{\tilde{k} \in \tilde{U_a^l} }(T(\tilde{k})-I(\tilde{k})) \leq \sup_{\tilde{k} \in \tilde{U_a^l} }(a+\epsilon- I(\tilde{k})) = a+\epsilon - \inf_{\tilde{k} \in \tilde{U_a^l}}{I(\tilde{k})}.
\]
This results in, 
\begin{align*}
\liminf_{n \rightarrow \infty}{\psi_n} &\geq -\epsilon + \sup_{a \in r_l}\sup_{\tilde{k} \in \tilde{U_a^l}}{(T(\tilde{k}) - I(\tilde{k}))} \\
& = -\epsilon + \sup_{\tilde{k} \in \mathcal{K}^l}{(T(\tilde{k}) - I(\tilde{k}))}
\end{align*}
Now we take $l \rightarrow \infty$ to get,
\begin{equation}\label{leftside}
\liminf_{n \rightarrow \infty}{\psi_n} \geq -\epsilon + \limsup_{l \rightarrow \infty}\sup_{\tilde{k} \in \tilde{\mathcal{K}^l}}{(T(\tilde{k}) - I(\tilde{k}))} 
\end{equation}

Since $\epsilon$ is arbitrary, combining (\ref{rightside}) and (\ref{leftside}) we have our theorem.

\subsection{Proof of Theorem \ref{ermmconcentration}}

\begin{proof}
Fix $\eta>0$. Define 
\[
\tilde{A} =  \{\tilde{k} \in \mathcal{\tilde{K}}: \delta(\tilde{k}, \tilde{F}^*)\geq \eta \}.
\]
By our assumption the maximizers of the function $T(\tilde{k}) - I(\tilde{k})$ are in $\tilde{\mathcal{K}}^{l_0}$ and we also have $\tilde{\mathcal{K}}^{l_0}$ is compact and $T(\tilde{k}) - I(\tilde{k})$ is upper semi-continuous in $\tilde{\mathcal{K}}^{l_0}$ (since $T$ is continuous in $\tilde{\mathcal{K}}^{l_0}$). Thus the set of maximizers $F^*$ is closed set and this implies $\tilde{A}$ is also a closed set. Next we introduce the following quantity,
 
\begin{align*}
\gamma &= \sup_{\tilde{k} \in \mathcal{\tilde{K}}}{(T(\tilde{k})-I(\tilde{k}))} - \sup_{\tilde{k} \in \tilde{A}}{(T(\tilde{k})-I(\tilde{k}))}  \\
&= \sup_{\tilde{k} \in \mathcal{\tilde{K}}^{l_0}}{(T(\tilde{k})-I(\tilde{k}))} - \sup_{\tilde{k} \in \tilde{A}}{(T(\tilde{k})-I(\tilde{k}))} >0.
\end{align*}
The last display follows from the fact that $\tilde{A}$ is a closed set disjoint from the set of maximizers(which is also a closed set). Fixed $l\geq 1$ and $\eps>0$ and from the proof of Theorem \ref{normconst}, recall the finite set $r_l$ and the cover $\cup_{a\in r_l} \tilde{C}_a^l$ of $\tilde{\cK}^l$. 
 Now we estimate the probability $R_n(\tilde{k} \in \tilde{A})$ by, 
\begin{align*}
R_n(\tilde{k} \in \tilde{A}) &= e^{-n^2 \psi_n} \left[ \int_{\tilde{A} \cap \{T(\tilde{k}) - T(f_l(\tilde{k})) > \epsilon\}} e^{n^2T(\tilde{k})} \, d\tilde{Q}_n (\tilde{k}) + \int_{\tilde{A} \cap \{T(\tilde{k}) - T(f_l(\tilde{k})) \leq \epsilon\}} e^{n^2T(\tilde{k})} \, d\tilde{Q}_n (\tilde{k}) \right] \\
& \leq  e^{-n^2 \psi_n} \int_{\tilde{A} \cap \{T(\tilde{k}) - T(f_l(\tilde{k})) > \epsilon\}} e^{n^2T(\tilde{k})} \, d\tilde{Q}_n(\tilde{k})+ e^{-n^2\psi_n} |r_l| \sup_{a \in r_l} {e^{n^2(a + 2\epsilon)}\tilde{Q}_n f_l^{-1}(\tilde{C_a^l} \cap \tilde{A})}
\end{align*}
We will assume that $\tilde{C_a^l} \cap \tilde{A}$ is non empty for each $a$, if not, we will just drop the sets. Now noting that $\tilde{C_a^l} \cap \tilde{A}$ are closed sets, using \eqref{ldpclosed} and Theorem \ref{normconst} we have,
\begin{align*}
&\limsup_{n \rightarrow \infty}{\frac{1}{n^2} \log{R_n(\tilde{k} \in \tilde{A})}} \\
&\leq \max\left(\limsup_{n \rightarrow \infty}\f{1}{n^2} \ln\int_{ \{T(\tilde{k}) - T(f_l(\tilde{k})) > \epsilon\}} e^{n^2T(\tilde{k})} \, d\tilde{Q}_n(\tilde{k})  ,\sup_{a \in r_l}(a+ 2\epsilon - \inf_{\tilde{k}\in \tilde{C_a} \cap \tilde{A}}{I_l(\tilde{k})})\right) - \sup_{\tilde{k} \in \mathcal{\tilde{K}}^l}{(T(\tilde{k}) - I(\tilde{k}))} 
\end{align*}
Now each $\tilde{k} \in \tilde{C_a^l} \cap \tilde{A}$ satisfies $T(\tilde{k}) \geq a$ and by Lemma \ref{lemteclem1} gives $I_l(k) \geq I(k) -\epsilon$ for large enough $l$, thus,
\[
\epsilon + \sup_{\tilde{k} \in \tilde{C_a^l} \cap \tilde{A}}(T(\tilde{k}) - I(\tilde{k})) \geq\sup_{\tilde{k} \in \tilde{C_a^l} \cap \tilde{A}}(T(\tilde{k}) - I_l(\tilde{k})) \geq  a -\inf_{\tilde{k} \in \tilde{C_a^l} \cap \tilde{A}}{I_l(\tilde{k})}
\]
Now combining the last two display and using the assumption \ref{bigneglect} or \ref{bigneglect1}  we have,
\begin{align*}
\limsup_{l \rightarrow \infty}\limsup_{n \rightarrow \infty}{\frac{1}{n^2} \log{R_n(k \in \tilde{A})}} &\leq 3\epsilon + \limsup_{l \rightarrow \infty}\left[ \sup_{a \in r_l}\sup_{\tilde{k} \in \tilde{C_a} \cap \tilde{A}}(T(\tilde{k}) - I(\tilde{k})) - \sup_{\tilde{k} \in \mathcal{\tilde{K}}^l}{(T(\tilde{k}) - I(\tilde{k}))} \right]\\
& = 3\epsilon +\limsup_{l \rightarrow \infty}\left[ \sup_{\tilde{k} \in  \tilde{A} \cap \mathcal{\tilde{K}}^l }(T(\tilde{k}) - I(\tilde{k})) - \sup_{\tilde{k} \in \mathcal{\tilde{K}}^l}{(T(\tilde{k}) - I(\tilde{k}))}\right] \\
&\le 3\epsilon - \gamma.
\end{align*}
Since $\epsilon$ is arbitrary, we let $\epsilon$ go to zero and the proof is complete.
\end{proof}

\subsection{Proof of Theorem \ref{thm:spec-gergm}:}

By Holder's inequality (the holder inequality can be used as we are integrating w.r.t Lebesgue measure over $[0,1]$) and using the assumption that either $k$ is non-negative or $e(H_i)$'s are positive even integers we have for each $2\leq i \leq s$,

\begin{equation*}
t(H_i, k) \leq  \int_{[0,1]^2} k(x_{1}, x_{2})^{e(H_i)} \, dx_1 dx_2.
\end{equation*}

Hence for non-negative $\beta_2,\ldots,\beta_s$
\begin{align*}
T(k) &\leq \beta_1 t(H_1, k) + \sum_{i=2}^s \beta_i \int_{[0,1]^2} k(x_{1}, x_{2})^{e(H_i)} \, dx_1 dx_2 \\
&= \int_{[0,1]^2} \sum_{i=1}^s \beta_i  k(x_{1}, x_{2})^{e(H_i)} \, dx_1 dx_2 
\end{align*}
At this end we have the following
\[
\sup_{k \in \mathcal{{K}}}{(T(k) - I(k))} \leq \sup_{u}{(\sum_{i=1}^s \beta_i u^{e(H_i)} - I(u))}
\] 
Again note that equality in Holder's inequality holds if $k$ is a constant function. Hence the inequality in the above display is in fact an equality Proving \eqref{stdnormconst}. It can be shown that the the constant functions are the only maximizers by the same argument as in \cite{cha}, we omit the details.

Now to prove the second part note that by \eqref{finitemaxermm} we have all maximizers of $(\sum_{i=1}^s \beta_i u^{e(H_i)} - I(u))$ are in a finite interval of the form $[-l, l]$.  Since $I(u)$ is convex function and $\sum_{i=1}^s \beta_i u^{e(H_i)}$ is a polynomial hence we have $\sum_{i=1}^s \beta_i u^{e(H_i)} - I(u)$ will have finitely many maximizers in a compact interval. Now \eqref{finitemaxermm} allows us to use Theorem \ref{ermmconcentration} and we conclude \eqref{hightemp}.

\subsection{Proof of Theorem \ref{thm:partfuncont}}

We start with the following elementary Lemma from large deviations. We provide this for completeness.

\begin{lem}\label{teclemnotrans}
Let $q_{ij}$'s are supported on the whole real line and the moment generating function $M(\theta)$ is continuous for all $\theta \in \bR$. Then the Legendre transform $h(.)$ is a finite continuous function on the whole real line.
\end{lem}
\begin{proof}
First fix a number $C>0$. Denote the measure corresponding to $q_{ij}$ by $\mu$. Then we have,
\begin{align}\label{teclemnotranseq1}
\liminf_{\theta \rightarrow \infty}{\f{\ln{M(\theta)}}{\theta}} &= \liminf_{\theta \rightarrow \infty}{\f{\ln {\int_{\bR}{e^{\theta x}} \,\mu(dx)}}{\theta}} \notag\\
&\geq \liminf_{\theta \rightarrow \infty}{\f{\ln{\int_{C}^{\infty}{e^{\theta x}} \,\mu(dx)}}{\theta}}  \notag \\
&\geq \liminf_{\theta \rightarrow \infty}{\f{\ln{\int_{C}^{\infty}{e^{\theta C}} \,\mu(dx)}}{\theta}} \notag \\
&\geq C + \liminf_{\theta \rightarrow \infty} \f{\ln{ \mu[C, \infty)}}{\theta} \notag \\
&= C \qquad \text{(since the measure is supported on real line)}.
\end{align}
Since in \eqref{teclemnotranseq1} $C>0$ is arbitrary, 
\[ \liminf_{\theta \rightarrow \infty}{\f{\ln{M(\theta)}}{\theta}} = \infty,\]
and similarly we also have,
\[ \liminf_{\theta \rightarrow -\infty}{-\f{\ln{M(\theta)}}{\theta}} = \infty.\]
Thus we have for each $x\in \bR$,
\[\lim_{\theta \rightarrow \infty} \left(\theta x -\ln M(\theta) \right) = \lim_{\theta \rightarrow \infty} \theta \left( x - \f{\ln M(\theta)}{\theta} \right) = -\infty,\] and
\[\lim_{\theta \rightarrow -\infty} \left(\theta x -\ln M(\theta) \right) = \lim_{\theta \rightarrow -\infty} \theta \left( x - \f{\ln M(\theta)}{\theta} \right) = -\infty .\]
Combining the last two display,
\[\lim_{|\theta| \rightarrow \infty} \left(\theta x -\ln M(\theta) \right) = -\infty.\]
Here we use the fact that $\ln M(\theta)$ is convex and hence the supremum of ${\left(\theta x -\ln M(\theta) \right)}$ is attained at some finite $\theta_{max} = \theta_{max}{(x)}$. This in particular gives $h(x) = \sup_{\theta \in \bR}{\left(\theta x -\ln M(\theta) \right)}= {\left(\theta_{max} x -\ln M(\theta_{max}) \right)} < \infty $ for all $x\in \bR$. Lastly, Legendre transform of convex function ($\log$ of MGF) is convex and finite convex function on $\bR$ is continuous. The proof is complete.
\end{proof}

Now we start the proof of Theorem \ref{thm:partfuncont}.

\begin{proof}
Fix $\mvbeta^0 :=(\beta_1^0,\beta_2^0,\ldots, \beta_s^0) \in B$. Since $\beta_2^0,\ldots, \beta_s^0$ are non negative real numbers then by Theorem \ref{thm:spec-gergm} we have $\lim_{n \rightarrow \infty}{\psi_n} = \sup_{u }{\left(\sum_{i=1}^s \beta_i u^{e(H_i)} - I(u)\right)}$. Consider a finite open set $O_ {\mvbeta^0} \subset B$ such that $\mvbeta^0 \in O_ {\mvbeta^0}$. Now under the assumption \eqref{finitemaxermm} we have,
\begin{equation}\label{uniformcompact}
\lim_{|u| \rightarrow \infty}\sup_{\mvbeta \in O_ {\mvbeta^0}}{\left(\sum_{i=1}^s \beta_i u^{e(H_i)} - I(u)\right)} \rightarrow -\infty.
\end{equation}
 Hence there is a compact set $C(\mvbeta^0)$ such that for any $\mvbeta \in O_ {\mvbeta^0}$,
\begin{equation}\label{uniformcompact1}
\sup_{u}\left(\sum_{i=1}^s \beta_i u^{e(H_i)} - I(u)\right) = \sup_{u \in C(\beta_0)}\left(\sum_{i=1}^s \beta_i u^{e(H_i)} - I(u)\right).
\end{equation}
Now since the base measure is supported on the real line, Lemma \ref{teclemnotrans} gives $I(u)$ is continuous and a continuous function on compact interval is uniformly continuous; the function $l(\beta_1,\ldots,\beta_s,u) := \left(\sum_{i=1}^s \beta_i u^{e(H_i)} - I(u)\right)$ is uniformly continuous in $u \in C(\beta_0)$ for each $\beta \in O_{\beta_0}$, proving our assertion. 
\end{proof}

\subsection{Proof of Theorem \ref{continhightemp}:} 
Let $\mvbeta \in O_ {\mvbeta_0}$(a finite open set) such that $\bar{O}_ {\beta_0} \subset B$. From the proof of Theorem \ref{thm:partfuncont} there exists a compact set $C(\mvbeta^0) \subset \bR$, such that for each $\mvbeta \in O_ {\mvbeta_0} $,  $\sup_{u}l(\mvbeta,u) = \sup_{u \in C(\mvbeta_0)}l(\mvbeta,u)$.  First by the compactness of $C(\mvbeta_0)$  we have, for each $\mvbeta \in C(\mvbeta_0)$, $\sup_{u \in C(\beta_0)}l(\mvbeta,u)$ is attained.  Denote the maximizer by $u^*(\mvbeta)$. Also let $M:= l_{max}(\mvbeta^0) = l(\beta_0,u^*(\mvbeta^0)) $. Fix $\epsilon > 0$. Since the maximizer is unique,  
\[
l(\mvbeta^0,u) <M  \quad \text{if} \quad |u - u^*(\mvbeta^0)| \geq \epsilon.
\]
Now by continuity of $l$ in $u$ we can choose two numbers $r$ and $s$ such that,
\[
l(\mvbeta^0,u)<r<s<M \qquad \text{if} \quad |u - u^*(\mvbeta^0)| \geq \epsilon.
\]
Now choose $\delta>0$ so that, the set of $\mvbeta$ such that $||\mvbeta -\mvbeta^0|| <\delta$ will be inside $O_ {\mvbeta^0}$. Further $l$ is uniformly continuous in $\bar{O}_ {\beta_0} \times \bR$ (by compactness of $\bar{O}_ {\beta_0}$ and Remark \ref{rmk:thm:partfuncont}), thus, 
\begin{equation}\label{teclemgenophasetran1}
l(\mvbeta,u) \leq r \quad \text{if}\quad ||\mvbeta -\mvbeta^0|| <\delta \quad \text{and}\quad|u - u^*(\mvbeta^0)| \geq \epsilon.
\end{equation}
Again continuity of $l$ at $(\mvbeta^0, u^*(\mvbeta^0))$ gives for some $\delta' < \delta$,
\begin{equation}\label{teclemgenophasetran2}
l(\mvbeta, u^*(\mvbeta^0)) > s \quad \text{if}\quad ||\mvbeta -\mvbeta^0|| <\delta'.
\end{equation}
For $\mvbeta$ such that $||\mvbeta -\mvbeta^0|| <\delta'$, from \eqref{teclemgenophasetran2} we get $l_{max}(\mvbeta)$ is at least $s$ and \eqref{teclemgenophasetran1} gives maximum cannot be attained on $|u - u^*(\mvbeta^0)| \geq \epsilon$, this maximum is attained on $|u - u^*(\mvbeta^0)| < \epsilon$. We can rephrase this as follows: given $\epsilon>0$ we can get $\delta'>0$ such that $|u^*(\mvbeta) - u^*(\mvbeta^0)| < \epsilon$ whenever $||\mvbeta -\mvbeta^0|| <\delta'$.  Continuity of $u^*(.)$ at $\mvbeta^0$ follows. 

\subsection{Proof of Theorem \ref{jstar}:}

 First note that $t(H_j, k)$ can be written as the following,
\begin{align*}
t(H_j, k) &= \int{\prod_{k =2}^j k(x_1,x_k)} \, \prod_{k=1}^j dx_k  \\
&= \int{F(x)^j} \, dx,
\end{align*}
where
\[
F(x) = \int{k(x,y)}\, dy.
\]
Now since $h$(defined in \eqref{con2A2}) is convex,
\[
\int{h(k(x,y))} \, dy \geq h(F(x)),
\]
hence
\[
I(k) \geq \int {I(F(x))} \, dx.
\]
and equality holds iff $k(x,y)$ is a constant function of almost (Lebesgue) all $y$. If we write,
\[
P(u) =\sum_{j=1}^s \beta_j u^j ,
\]
thus we have,
\[
T(k) -I(k) = \int{P(F(x))} \, dx -I(k) \leq  \int{\left(P(F(x)) - I(F(x)) \right)} \, dx
\]
by the discussion above the equality in the last display holds when $k(x,y)$ is a constant function of $y$ for almost all $x$ and $M(x)$ equals a constant that maximizes $P(u) -I(u)$. Since $k$ is symmetric we must have $k$ a constant function by the first condition. The rest of the proof is similar to Theorem \ref{thm:spec-gergm}, we omit the details.

\subsection{Proof of Theorem \ref{thm:normal-full-solvable}: }

\begin{proof}
Note that the normalizing constant is the expectation 
\[\psi_n = \mathbb{E}(e^{{\beta_1} \sum_{i,j}{x_{ij}} + \f{\beta_2}{n}\sum_{i}(\sum_{j=1}^n x_{ij})^2 }),\] 
where $x_{ij}$'s are i.i.d. normal distributed with mean zero and variance one for $i<j$ and $x_{ij}=x_{ji}$ with $x_{ii} = 0$. We can write the expectation as, 
\begin{equation}\label{lemnormaleq1}
\mathbb{E}(e^{{\beta_1} \sum_{i,j}{x_{ij}} + \f{\beta_2}{n}\sum_{i}(\sum_{j=1}^n x_{ij})^2 }) = \mathbb{E}(e^{{\beta_1} \sum_{i,}{y_{i}} + \f{\beta_2}{n}\sum_{i}(y_{i})^2 }) = \mathbb{E}(\exp{({\beta_1} \bold{1}'\mvy + \f{\beta_2}{n} \mvy' \mvy)}),
\end{equation}
where $\mvy' = (y_1,\ldots, y_n)$ have a multivariate distribution with mean zero and variance matrix $\Sigma = (\sigma_{ij})_{1\leq i,j\leq n}$ with $\sigma_{ii} = n-1$ and $\sigma_{ij} = 1$ if $i \neq j$. Now define $\mvz = \Sigma^{-\f{1}{2}}\mvy$, clearly $\mvz$ is $n$-variate standard normal distribution (i.e. with independent components). 
Now the exponent in the \eqref{lemnormaleq1} $L(\mvy) := {\beta_1} \bold{1}'\mvy + \f{\beta_2}{n} \mvy' \mvy$ can be re-written as 
\begin{align*}
 L(\mvy) = {\beta_1} \bold{1}'\mvy + \f{\beta_2}{n} \mvy' \mvy &=  {\beta_1} \bold{1}'(\Sigma^{\f{1}{2}}\mvz)+ \f{\beta_2}{n} (\Sigma^{\f{1}{2}}\mvz)' (\Sigma^{\f{1}{2}}\mvz)  \\
 &= {\beta_1} \bold{1}'(\Sigma^{\f{1}{2}}\mvz)+ \f{\beta_2}{n} \mvz'  \Sigma \mvz.
 \end{align*}
 Using the spectral decomposition of $\Sigma = P' \Lambda P$ we have,
 \begin{align*}
{\beta_1} \bold{1}'(\Sigma^{\f{1}{2}}\mvz)+ \f{\beta_2}{n} \mvz'  \Sigma \mvz = {\beta_1} \bold{1}'(P' \Lambda^{\f{1}{2}} P \mvz)+ \f{\beta_2}{n} \mvz'  P' \Lambda P \mvz.
\end{align*}
Since $P$ is orthogonal $\mvu = P \mvz$ also has an $n$-variate standard normal distribution. Further the term in the exponent expressed in terms of $\vu$ can be written as,
\begin{align*}
L(\mvy) = {\beta_1} \bold{1}'(P' \Lambda^{\f{1}{2}} \mvu)+ \f{\beta_2}{n} \mvu' \Lambda \mvu.
\end{align*}
At this end we note that the largest Eigen value of $\Sigma$ is $\lambda_1 = 2(n-1)$ with Eigen vector $\bold{1}' = (1,\ldots,1)$ and all other Eigen values equal to $\lambda_i = n-2$ for $i = 2,\ldots,n$. This further implies the first row of $P$ is $(\f{1}{\sqrt{n}}, \ldots, \f{1}{\sqrt{n}})$ and the orthogonality of $P$ gives $P\bold{1} = (\sqrt{n},0,\ldots,0)'$. Finally,
\begin{align*}
L(\mvy) &= {\beta_1} \bold{1}'(P' \Lambda^{\f{1}{2}} \mvu)+ \f{\beta_2}{n} \mvu' \Lambda \mvu \\
&=  {\beta_1}\sqrt{2n(n-1)}u_1 + \f{2\beta_2}{n} {(n-1)} u_1^2 + \f{\beta_2}{n} {(n-2)}\sum_{i=2}^{n} {u_i^2}
\end{align*}
Using the last calculation finally we can write \eqref{lemnormaleq1} as weighted sum of i.i.d normal random variable,
\begin{align}
\psi_n &=  \mathbb{E}(\exp{({\beta_1}\sqrt{2n(n-1)}u_1 + \f{2\beta_2}{n} (n-1) u_1^2)}) \prod_{i=2}^n {\mathbb{E}(\exp{(\f{\beta_2}{n} (n-2) {u_i^2})})} \notag \\
&=\f{1}{\sqrt{1-\f{4\beta_2(n-1)}{n}}} \exp{\left(\f{\beta_1^2 n(n-1)}{1-4\beta_2\f{(n-1)}{n}}\right)} \prod_{i=2}^n \f{1}{\sqrt{1-\f{2\beta_2(n-2)}{n}}}\notag \\
&= \f{1}{\sqrt{1-\f{4\beta_2(n-1)}{n}}} \exp{\left(\f{\beta_1^2 n(n-1)}{1-4\beta_2\f{(n-1)}{n}}\right)} \left( {1-\f{2\beta_2(n-2)}{n}} \right)^{-\f{(n-1)}{2}}. \label{techline1lemnorm} 
\end{align}
The integral in \eqref{techline1lemnorm} does not exist when $\beta_2 \geq \f{n}{4(n-1)}$. This completes the proof.  
\end{proof}

\subsection{Proof of Theorem \ref{thm:edge-two-star-four}: }

Recall the expression of $T(\cdot)$ from \eqref{eqn:edge-two-four}.  
 Now we have
\begin{align}
|T(\mvx) - T(\mvx^l)| 
&\le \f{|\beta_2|}{n^3}| \sum_{i=1}^n (\sum_{j=1}^n x_{ij} - \sum_{j=1}^n x_{ij}^l + \sum_{j=1}^n x_{ij}^l)^2 -\sum_{i=1}^n (\sum_{j=1}^n x_{ij}^l)^2| + \f{|\beta_1|}{n^2} \sum_{i,j}{|x_{ij} - x_{ij}^l|}\notag \\
&= \f{|\beta_2|}{n^3}\left| \sum_{i=1}^n\left( (\sum_{j=1}^n x_{ij} - \sum_{j=1}^n x_{ij}^l)^2 + (\sum_{j=1}^n x_{ij}^l)^2 + 2(\sum_{j=1}^n x_{ij}^l)(\sum_{j=1}^n x_{ij} - \sum_{j=1}^n x_{ij}^l)\right) -\sum_{i=1}^n (\sum_{j=1}^n x_{ij}^l)^2\right| \notag\\
& + \f{|\beta_1|}{n^2} \sum_{i,j}{|x_{ij} - x_{ij}^l|}\notag 
\end{align}
Simple algebraic manipulations then results in,
\begin{align}
|T(\mvx) - T(\mvx^l)|&\leq \f{|\beta_2|}{n^3}\left| \sum_{i=1}^n (\sum_{j=1}^n x_{ij} - \sum_{j=1}^n x_{ij}^l)^2 + 2 \sum_{i=1}^n (\sum_{j=1}^n x_{ij}^l)(\sum_{j=1}^n x_{ij} - \sum_{j=1}^n x_{ij}^l) \right| + \f{|\beta_1|}{n^2} \sum_{i,j}{|x_{ij} - x_{ij}^l|}\notag 
\end{align}
Now using Cauchy-Schwartz and $|x^l|\leq l$ for all $x\in \bR$ results in, 
\begin{align}
|T(\mvx) - T(\mvx^l)| &\le \f{|\beta_2|}{n^3} \left| \sum_{i=1}^n n\sum_{j=1}^n (x_{ij} - x_{ij}^l)^2 + 2 \sum_{i=1}^n (\sum_{j=1}^n x_{ij}^l)(\sum_{j=1}^n x_{ij} - \sum_{j=1}^n x_{ij}^l) \right| + \f{|\beta_1|}{n^2} \sum_{i,j}{|x_{ij} - x_{ij}^l|}\notag \\
& \le \f{|\beta_2|}{n^2} \sum_{i=1}^n \sum_{j=1}^n (x_{ij} - x_{ij}^l)^2 +  \f{(2l| \beta_2| + |\beta_1|)}{n^2} \sum_{i=1}^n \sum_{j=1}^n |x_{ij} - x_{ij}^l|. \label{formarkovest}
\end{align}
Now using the calculation above we have for all $M>0$,
\begin{align}
\cR_n&:= C_4^{n \choose 2}\int_{|T(\mvx)-T(\mvx^l)| > \epsilon}\exp{\left({\beta_1} \sum_{i,j}{x_{ij}} + \f{\beta_2}{n}\sum_{i,j,k}{x_{ij} x_{ik}}- \sum_{i<j}x_{ij}^4\right)} \,\prod_{i<j}dx_{ij} \notag \\
& =C_4^{n \choose 2} \int_{M n^2 |T(\mvx)-T(\mvx^l)| > M n^2\epsilon}\exp{\left({\beta_1} \sum_{i,j}{x_{ij}} + \f{\beta_2}{n}\sum_{i,j,k}{x_{ij} x_{ik}}- \sum_{i<j}x_{ij}^4\right)} \,\prod_{i<j}dx_{ij} \notag \\
&\leq C_4^{n \choose 2} e^{-M n^2\epsilon}\int \exp{\left(n^2M|T(\mvx)-T(\mvx^l)|  + {\beta_1} \sum_{i,j}{x_{ij}} + \f{\beta_2}{n}\sum_{i,j,k}{x_{ij} x_{ik}}-  \sum_{i<j}x_{ij}^4\right)} \,\prod_{i<j}dx_{ij}.\notag\\
&:= C_4^{n \choose 2} e^{-M n^2\epsilon} \cE_n , \qquad \text{say.}\label{est1}
\end{align}

Using \eqref{formarkovest} and writing $\lambda(d\vx) = \prod_{i<j}dx_{ij} $,
\begin{align*}
\cE_n &\le \int \exp{(M{|\beta_2|} \sum_{i,j} (x_{ij} - x_{ij}^l)^2 +  {M(2l| \beta_2| + |\beta_1|)} \sum_{i,j} |x_{ij} - x_{ij}^l|  + {\beta_1} \sum_{i,j}{x_{ij}} + {\beta_2}\sum_{i,j}{x_{ij}^2}-  \sum_{i<j}x_{ij}^4)} \,\lambda(d\vx)\\
&= \int\exp(\sum_{i< j} \cL(x_{ij}))\lambda(d\vx),
\end{align*}
where $\cL:\bR\to\bR$ is given by,
\[\cL(x):= 2M{|\beta_2|} (x - x^l)^2 + 2M(2l| \beta_2| + |\beta_1|)|x - x^l| + 2\beta_1 x + 2\beta_2 x^2-x^4.  \]
This implies that with $\cR_n$ as in \eqref{est1},
\begin{align*}
 \cR_n &\leq  e^{-M n^2\epsilon} 
 \left(C_4 \int \exp{(2M{|\beta_2|}  (x - x^l)^2 +  {2M(2l| \beta_2| + |\beta_1|)}  |x - x^l|  + {2\beta_1} {x} + {2\beta_2}{x^2}-  x^4)} \, dx\right)^{n \choose 2}.
\end{align*}

Now note that,
\begin{align}\label{inthreeparts}
 &  \int \exp{(2M{|\beta_2|}  (x - x^l)^2 +  {2M(2l| \beta_2| + |\beta_1|)}  |x - x^l|  + {2\beta_1} {x} + {2\beta_2}{x^2}-  x^4)} \, dx \notag \\
 &=  \int_{x < -l} \exp{(2M{|\beta_2|}  (x + l)^2 -  {2M(2l| \beta_2| + |\beta_1|)}  (x + l)  + {2\beta_1} {x} + {2\beta_2}{x^2}-  x^4)} \, dx \notag \\
 &+ \int_{-l \le x \le l} \exp{( {2\beta_1} {x}+ {2\beta_2}{x^2}-  x^4)} \, dx \notag \\
 & +  \int_{ x > l} \exp{(2M{|\beta_2|}  (x - l)^2 +  {2M(2l| \beta_2| + |\beta_1|)}  (x - l)  + {2\beta_1} {x} + {2\beta_2}{x^2}-  x^4)} \, dx .
 \end{align}
 Consider the third term. Simplifying the exponent we get,
 \begin{align}\label{partthree}
 &\int_{ x > l} \exp{(2M{|\beta_2|}  (x - l)^2 +  {2M(2l| \beta_2| + |\beta_1|)}  (x - l)  + {2\beta_1} {x} + {2\beta_2}{x^2}-  x^4)} \, dx  \notag\\
 &= \exp(-2M|\beta_2|l^2 - 2M|\beta_1|l) \int_{ x > l} \exp(2M|\beta_2|x^2 +2M|\beta_2|x +  {2\beta_1} {x} + {2\beta_2}{x^2}-  x^4) \, dx \notag \\
 &\le \exp(-2M|\beta_2|l^2 - 2M|\beta_1|l) C(M).
 \end{align}
 for a constant $C(M)$. Enlarging the constant we can similarly have the following upper bound on the first term of \eqref{inthreeparts},
\begin{align}\label{partone}
& \int_{x < -l} \exp{(2M{|\beta_2|}  (x + l)^2 -  {2M(2l| \beta_2| + |\beta_1|)}  (x + l)  + {2\beta_1} {x} + {2\beta_2}{x^2}-  x^4)} \, dx \notag \\
&\le \exp(-2M|\beta_2|l^2 - 2M|\beta_1|l) C(M).
\end{align}

Now using \eqref{est1}, \eqref{partthree}, \eqref{partone} and \eqref{inthreeparts} we get
\begin{align*}
&\Delta_{n,l} := C_4^{n \choose 2}\int_{|T(\mvx)-T(\mvx^l)| > \epsilon}\exp{({\beta_1} \sum_{i,j}{x_{ij}} + \f{\beta_2}{n}\sum_{i,j,k}{x_{ij} x_{ik}}- \sum_{i<j}x_{ij}^4)} \,\prod_{i<j}dx_{ij} \\
&\le e^{-M n^2\epsilon} \left( 2C_4 C(M)\exp(-2M|\beta_2|l^2 - 2M|\beta_1|l)  +  C_4\int \exp{( {2\beta_1} {x}+ {2\beta_2}{x^2}-  x^4)} \, dx \right)^{n \choose 2} 
\end{align*}
Thus we have,
\begin{align*}
&\limsup_{n \rightarrow \infty}\f{1}{n^2}\ln{\Delta_{n,l}} \\
&\le -M \epsilon + \f{1}{2}\ln\left( 2C_4 C(M)\exp(-2M|\beta_2|l^2 - 2M|\beta_1|l)  +  C_4\int \exp{( {2\beta_1} {x}+ {2\beta_2}{x^2}-  x^4)} \, dx \right),
\end{align*}
and letting $l \rightarrow \infty$,
\begin{equation}\label{estf}
\limsup_{l \rightarrow \infty}\limsup_{n \rightarrow \infty}\f{1}{n^2}\ln{\Delta_{n,l}} \le -M \epsilon + \f{1}{2}\ln\left(C_4\int \exp{( {2\beta_1} {x}+ {2\beta_2}{x^2}-  x^4)} \, dx \right).
\end{equation}
Since \eqref{estf} is true for all $M>0$, we let $M$ go to infinity and we have the result. 
\qed

\subsection{Proof of Theorem \ref{contnodeweight}:}

Fix $\epsilon> 0$.  We can approximate  $t(F, \alpha_n, k_n) - t(F, \alpha, k)$ as follows: using uniform convergence of $\alpha_n \rightarrow \alpha$, $\exists N(\eps)>0$ such that uniformly on $[0,1]^{V(F)}$, for all $n> N(\eps)$,  we have $ |\prod_{i \in V(F)}{\alpha_n(x_i)^{\alpha_{i}(F)}} - \prod_{i \in V(F)}{\alpha(x_i)^{\alpha_{i}(F)}}| < \epsilon$. 
Now 
\begin{align}\label{thmcontterm0}
& t(F, \alpha_n, k_n) - t(F, \alpha, k) = \notag \\
& \int_{[0, 1]^{|V(F)|}}\left( \prod_{i \in V(F)}{\alpha_n(x_i)^{\alpha_{i}(F)}} \prod_{(i,j) \in E(F)}k_n(x_i, x_j)- \prod_{i \in V(F)}{\alpha(x_i)^{\alpha_{i}(F)}} \prod_{(i,j) \in E(F)}k(x_i, x_j) \right)\, \prod_{i \in V(F)}{dx_i}  \notag\\
& = \int_{[0, 1]^{|V(F)|}}\left( \prod_{i \in V(F)}{\alpha_n(x_i)^{\alpha_{i}(F)}} - \prod_{i \in V(F)}{\alpha(x_i)^{\alpha_{i}(F)}}\right)\prod_{(i,j) \in E(F)}k_n(x_i, x_j) \, \prod_{i \in V(F)}{dx_i}   \notag \\
& + \int_{[0, 1]^{|V(F)|}}\prod_{i \in V(F)}{\alpha(x_i)^{\alpha_{i}(F)}} \left( \prod_{(i,j) \in E(F)}k_n(x_i, x_j)-  \prod_{(i,j) \in E(F)}k(x_i, x_j) \right)\, \prod_{i \in V(F)}{dx_i}
\end{align}
The first term in \eqref{thmcontterm0} goes to zero by uniform convergence of $\alpha_n$ to $\alpha$. Next we show that the second term also goes to zero by the convergence of $k_n$ to $k$ in cut-metric. Suppose $E(F) = \{e_1, e_2,\ldots,e_m\}$. For convenience suppose that  $i_t, j_t$ be the endpoints of the edge $e_t$. Using the ``Lindeberg'' trick we get,
\begin{align*}
&\int_{[0, 1]^{|V(F)|}}{ \prod_{i \in V(F)}{\alpha(x_i)^{\alpha_{i}(F)}} \left( \prod_{(i,j) \in E(F)}k_n(x_i, x_j) -   \prod_{(i,j) \in E(F)}k(x_i, x_j)\right)} \, \prod_{i \in V(F)}{dx_i} \\
&= \sum_{t=1}^m{\int_{[0,1]^{|V(F)|}}{ {\prod_{i \in V(F)}{\alpha(x_i)^{\alpha_{i}(F)}}}{\prod_{s<t}k_n(x_{i_s}, x_{j_s})}{\prod_{s>t}k(x_{i_s}, x_{j_s})}}\left( k_n(x_{i_t}, x_{j_t}) - k(x_{i_t}, x_{j_t})\right)} \, \prod_{i \in V(F)}{dx_i}.
\end{align*}

Now consider a term from the above sum. To simplify notation assume that $i_t = 1$ and $j_t = 2$. Now let $X(x_1, x_3,...,x_k)$ be the terms in 
${{\prod_{s<t}k_n(x_{i_s}, x_{j_s})}{\prod_{s>t}k(x_{i_s}, x_{j_s})}} {\prod_{i \in V(F)}{\alpha(x_i)^{\alpha_{i}(F)}}}$ that contain $x_1$ and $Y(x_2, x_3,...,x_k)$ denote the rest of the terms in that product. Thus we now have the following,
\begin{align}
&|{\int_{[0,1]^{|V(F)|}}{{\prod_{i \in V(F)}{\alpha(x_i)^{\alpha_{i}(F)}}}{\prod_{s<t}k_n(x_{i_s}, x_{j_s})}{\prod_{s>t}k(x_{i_s}, x_{j_s})}}\left(k_n(x_{i_t}, x_{j_t}) - k(x_{i_t}, x_{j_t})\right)} \, \prod_{i \in V(F)}{dx_i} | \notag \\
& = |\int_{[0,1]^{k-2}} \left( \int_{[0,1]^2} X(x_1, x_3,...,x_k)Y(x_2, x_3,...,x_k)\left(k_n(x_{1}, x_{2}) - k(x_{1}, x_{2}) \, dx_1 dx_2\right)\right) \, dx_3...dx_k |. \label{thmcontterm2} 
\end{align}

Finally we have $k_n$'s are uniformly bounded and $k_n$ converges to $k$ in cut-metric, hence combining \eqref{thmcontterm0}, \eqref{thmcontterm2} the theorem is proved.

\section*{Acknowledgements}
SB and SuC have been supported by NSF DMS-1613072, DMS-1606839 and ARO grant W911NF-17-1-0010. SB, SC and BD have been partially supported by NSF SES grant 1357622. SC has been partially  supported by NSF SES-1461493, and SES-1514750.  BD has been partially supported by NSF grants SES-1558661, SES-1619644,
SES-1637089, CISE-1320219, SMA-1360104.  We would like to thank Mathew Denny, James Wilson and Sayan Banerjee for illuminating discussions on the relevance of the results in this paper for applications. We thank for referees for reading the paper closely and providing many valuable  suggestions. 


\end{document}